%% file: main.tex
\documentclass[preprint]{elsarticle}
\input{header}
\usepackage{lineno}
\modulolinenumbers[5]

\makeatletter
\let\@afterindenttrue\@afterindentfalse
\def\ps@pprintTitle{%
	\let\@oddhead\@empty
	\let\@evenhead\@empty
	\def\@oddfoot{}%
	\let\@evenfoot\@oddfoot}
\makeatother

\begin{document}

\begin{frontmatter}

\title{Robust strategic planning for mobile medical units with steerable and unsteerable demands}

%% or include affiliations in footnotes:
\author[mymainaddress]{Christina B\"using}
\ead{buesing@math2.rwth-aachen.de}

\author[mymainaddress]{Martin Comis\corref{mycorrespondingauthor}}
\ead{comis@math2.rwth-aachen.de}
\cortext[mycorrespondingauthor]{Corresponding author}

\author[mysecondaryaddress]{Eva Schmidt}
\ead{eva.schmidt@mathematik.uni-kl.de}

\author[mysecondaryaddress]{Manuel Streicher}
\ead{streicher@mathematik.uni-kl.de}

\address[mymainaddress]{Lehrstuhl II f\"ur Mathematik, RWTH Aachen University, Germany}
\address[mysecondaryaddress]{Optimization Research Group, Technische Universit\"at Kaiserslautern, Germany}

\begin{abstract}
Mobile medical units (MMUs) are customized vehicles fitted with medical equipment that are used to provide primary care in rural environments.
As MMUs can be easily relocated, they enable a demand-oriented, flexible, and local provision of health services.
In this paper, we investigate the strategic planning of an MMU service by deciding where MMU operation sites should be set up and how often these should be serviced.
To that end, we study the strategic planning problem for MMUs ($\strategicPlanningProblem$) --  a capacitated set covering problem that includes existing practices and two types of patient demands: 
(i) steerable demands representing patients who seek health services through a centralized appointment system and can be steered to any treatment facility within a given consideration set and (ii) unsteerable demands representing walk-in patients who always visit the closest available treatment facility.
We propose an integer linear program for the 
$\strategicPlanningProblem$ that can be solved via Benders decomposition and constraint generation.
Starting from this formulation, we focus on the uncertain version of the problem in which steerable and unsteerable demands are modeled as random variables that may vary within a given interval.
Using methods from robust optimization and duality theory, we devise exact constraint generation methods to solve the robust counterparts for interval and budgeted uncertainty sets.
All our results transfer to the session-specific $\strategicPlanningProblem$ and we evaluate our models in a computational study based on a set of instances generated from a rural primary care system in Germany.
\end{abstract}

\begin{keyword}
strategic planning \sep mobile medical units \sep covering problem \sep demand uncertainty \sep robust optimization \sep  Benders decomposition
\MSC[2010]  05C69\sep 90B50 \sep 90B80 \sep 90C10 \sep 90C90
\end{keyword}

\end{frontmatter}

\input{intro}
\input{literature}
\input{prob_def}
\input{uncertainty_general}
\input{multi_period}
\input{computational_study}

\input{conclusion}
\bibliography{MyCollection}
\bibliographystyle{model5-names}
\include{appendix}

\end{document}

%% file: header.tex
\usepackage[utf8]{inputenc}
\usepackage[numbers]{natbib}
\usepackage[english]{babel}
\usepackage{amsmath}
\usepackage{amsthm}
\usepackage{thmtools}
\usepackage{thm-restate}
\usepackage{bm}
\usepackage{mathtools}
\usepackage{xspace}
\usepackage{booktabs}
\usepackage{multirow}
\usepackage{amssymb}
\usepackage[hidelinks]{hyperref}
\usepackage{todonotes}
\usepackage{marvosym}
\usepackage[short, no comma]{optidef}
\usepackage{caption}
\usepackage{float}
\usepackage[toc,page]{appendix}
\usepackage{ mathrsfs}
\usepackage[binary-units=true]{siunitx}
\pgfdeclareverticalshading{rainbow}{100bp}
{color(0bp)=(red); color(25bp)=(red); color(35bp)=(yellow);
	color(45bp)=(green); color(55bp)=(cyan); color(65bp)=(blue);
	color(75bp)=(violet); color(100bp)=(violet)}

\usepackage[a4paper,
totalwidth=.643\paperwidth,
totalheight=.93\paperheight,
headsep=.03\paperheight,
headheight = 13.59999pt,
hmarginratio=6:8,
vmarginratio=1:1,
includeheadfoot, 
bindingoffset=11mm,
dvips]{geometry}

\newcommand{\strategicPlanningProblem}{\mathrm{SPMMU}\xspace}
\newcommand{\robustStrategicPlanningProblem}{\mathrm{rSPMMU}\xspace}
\newcommand{\sessionSpecificStrategicPlanningProblem}{\mathrm{\session SPMMU}\xspace}

\newcommand{\location}{\ell}
\newcommand{\elocation}{\boldsymbol{\location}}
\newcommand{\Locations}{L}
\newcommand{\eLocations}{\boldsymbol{\Locations}}
\newcommand{\practice}{p}
\newcommand{\epractice}{\boldsymbol{\practice}}
\newcommand{\Practices}{P}
\newcommand{\ePractices}{\boldsymbol{\Practices}}
\newcommand{\demand}{v}
\newcommand{\edemand}{\boldsymbol{\demand}}
\newcommand{\Demands}{V}
\newcommand{\eDemands}{\boldsymbol{\Demands}}

\newcommand{\ConsiderationSet}{N}
\newcommand{\eSteerableConsiderationSet}{\boldsymbol{\ConsiderationSet}^\steerableDemand}
\newcommand{\eUnsteerableConsiderationSet}{\boldsymbol{\ConsiderationSet}^\unsteerableDemand }

\newcommand{\closestFacility}[2]{k_{#2}^{\min}(#1)}
\newcommand{\eclosestFacility}[2]{\boldsymbol{k}_{#2}^{\min}(#1)}
\newcommand{\closestFixedFacility}[2]{k_{#2}^{\min}(#1)}

\newcommand{\costLocation}{c}
\newcommand{\costSession}{\hat{c}}
\newcommand{\capacitySession}{\hat{b}}
\newcommand{\capacityPractice}{\bar{b}}
\newcommand{\capacityLocation}{b}

\newcommand{\steerableDemand}{d}
\newcommand{\varUnsteerableDemand}{w}
\newcommand{\unsteerableDemand}{u}
\newcommand{\helper}{\gamma}

\newcommand{\dualBudgetCap}{\varepsilon}

\newcommand{\varEncodeU}{o}
\newcommand{\varEncodeN}{n}
\newcommand{\varEncodeV}{r}

\newcommand{\N}{\mathbb{N}}
\newcommand{\Z}{\mathbb{Z}}
\newcommand{\R}{\mathbb{R}}

\newcommand{\Uncertainty}{\mathcal{U}}
\newcommand{\BudgetedUncertainty}{\Uncertainty^{\demandBound}}
\newcommand{\demandBound}{\Gamma}
\newcommand{\lowerSteerable}{\alpha}
\newcommand{\upperSteerable}{\beta}
\newcommand{\AllScenSteerable}{\Xi}
\newcommand{\scenSteerable}{\xi}
\newcommand{\indexSteerable}{1}

\newcommand{\lowerUnsteerable}{\sigma}
\newcommand{\upperUnsteerable}{\tau}
\newcommand{\AllScenUnsteerable}{H}
\newcommand{\scenUnsteerable}{\eta}
\newcommand{\indexUnsteerable}{2}
\newcommand{\MMUoperation}{m}
\newcommand{\assignmentSteerablePatientDemand}[1]{f_{#1}}
\newcommand{\SubsetOfSubsets}{\mathscr{U}}

\newcommand{\arcCapacity}{\mu}

\newcommand{\OriginalFormulation}{(\mathrm{Det})}
\newcommand{\OriginalSessionSpecificFormulation}{(\mathrm{\session Det})}
\newcommand{\BendersSubproblem}[3]{(\mathrm{SP})(#1,#2,#3)}
\newcommand{\LPBendersSubproblem}[3]{(\mathrm{SP}_{\mathrm{LP}})(#1,#2,#3)}
\DeclareMathOperator{\MasterProblem}{(MP)}
\DeclareMathOperator{\BendersFormulation}{(Det-B)}
\DeclareMathOperator{\BendersFormulationNB}{Det-B}
\DeclareMathOperator{\SessionSpecificBendersFormulation}{(\session Det-B)}
\DeclareMathOperator{\RobustBendersFormulation}{(Rob-B)}
\DeclareMathOperator{\IntervalRobustBendersFormulation}{(RobI-B)}
\DeclareMathOperator{\IntervalRobustBendersFormulationNB}{RobI-B}
\DeclareMathOperator{\GammaRobustBendersFormulation}{(Rob\demandBound-B)}
\DeclareMathOperator{\GammaRobustBendersFormulationNB}{Rob\demandBound-B}
\DeclareMathOperator{\SeparationProblemFormulation}{(Sep)}
\DeclareMathOperator{\SeparationProblemFormulationAppendix}{(Sep')}
\newcommand{\PrimalUnsteerable}[1]{(\mathrm{P}^U)(#1)}
\newcommand{\LPPrimalUnsteerable}[1]{(\mathrm{P}^U_{\mathrm{LP}})(#1)}
\newcommand{\DualUnsteerable}[1]{(\mathrm{D}^U_{\mathrm{LP}})(#1)}
\newcommand{\InducedNeighborhood}{\Demands^U(\hat{\varUnsteerableDemand})}

\newcommand{\LPPrimalAssumption}[1]{(\mathrm{P}^k_{\mathrm{LP}})(#1)}
\newcommand{\LPDualAssumption}[1]{(\mathrm{D}^k_{\mathrm{LP}})(#1)}

\DeclareMathOperator{\dist}{dist}

\newcommand{\maxDist}{\Delta}
\newcommand{\simulatedDemand}[2]{q^{#1}_{#2}}
\newcommand{\avgSimulatedDemand}[1]{\bar{q}_{#1}}
\newcommand{\round}[1]{\text{round}\hspace*{-1pt}\left( #1 \right) }
\newcommand{\roundFixBracket}[1]{\text{round}\bigg( #1 \bigg) }
\newcommand{\fracUnsteerable}{\omega}

\newcommand{\session}{t}
\newcommand{\Sessions}{\mathcal{T}}

\newcommand{\Order}[1]{\pi_{#1}}
\newcommand{\order}[2]{\Order{#1}(#2)}
\newcommand{\InvOrder}[1]{\pi^{-1}_{#1}}
\newcommand{\invorder}[2]{\InvOrder{#1}(#2)}
\newcommand{\eOrder}[1]{\boldsymbol{\pi}_{#1}}
\newcommand{\eorder}[2]{\eOrder{#1}(#2)}
\newcommand{\eInvOrder}[1]{\boldsymbol{\pi}^{-1}_{#1}}
\newcommand{\einvorder}[2]{\eInvOrder{#1}(#2)}

\theoremstyle{definition}
\newtheorem{definition}{Definition}

\newtheorem{assumption}{Assumption}

\theoremstyle{plain}
\newtheorem{theorem}[definition]{Theorem}

\newtheorem{corollary}[definition]{Corollary}
\newtheorem{lemma}[definition]{Lemma}

\usepackage{amstext} % for \text macro
\usepackage{array}   % for \newcolumntype macro
\newcolumntype{L}{>{$}l<{$}} % math-mode version of "l" column type
\newcolumntype{R}{>{$}r<{$}} % math-mode version of "l" column type

\definecolor{yellow}{HTML}{EFC000} % RWTH-blau
\definecolor{blue}{HTML}{0073C2} % RWTH-blau
\definecolor{gray}{HTML}{868686} % RWTH-blau

%% file: intro.tex
\section{Introduction}
\label{sec:introduction}
Health is a pillar of the prosperity and well-being of a society~\cite{dodd2005health}.
Therefore, the majority of nations world-wide implement health systems 
``whose primary purpose is to promote, restore and/or maintain 
health''~\cite{WHO_REPORT}.
Recently, the functioning of these systems has been threatened by ongoing changes in demographics~\cite{MAN10}.
In rural primary care, a decreasing number of physicians face an aging 
population with increasing needs that is scattered over sparsely populated 
regions~\cite{MAN10,doi:10.1111}.
While rural health has always suffered from a low density of health professionals and thus long access distances~\cite{10.2307/3762568}, this development poses the risk of multiplying the existing barriers to health services.

To counteract the growing distances between patients and services, the overcoming of access barriers with the help of mobile medical units (MMUs) has been studied in many developed and developing countries; compare~\cite{10.2307/3762568,THORSEN20181062,doi:10.1177/0972063417717900,Hill14,Schwartze2017}.
MMUs are customized vehicles fitted with medical equipment that are easy to relocate and can provide most of the health services a regular stationary practice could~\cite{doi:10.1177/0972063417717900}.
The flexibility of MMUs offers the possibility of a local and demand-oriented provision of primary care in sparsely populated regions, which are characteristic to rural settings~\cite{Hill14}.

Although MMUs have been reported to operate in various modes, we will focus exclusively on a weekly recurring operation in clinical sessions as described in~\cite{Schwartze2017,10.2307/3762568,THORSEN20181062}.
In this mode of operation, MMUs are stationed in larger cities (which they do not serve) and set out each day to provide health services at fixed sites in the surrounding rural communities.
As common in primary care, we thereby structure each day into a morning and an afternoon session and thus MMUs can service at most two sites per day~\cite{KLASSEN199683}.
At the end of each day, MMUs return to their home depot such that all personnel can return to their home overnight.
Next to the benefit of increased staff satisfaction, this incidentally reduces the cost for nonexempt staff~\cite{THORSEN20181062}.

The potentials of MMUs manifest in the increasing number of applications in practice.
The US alone has an estimated \num{1500} MMUs receiving \num{5} million visits each year~\cite{Hill14}.
However, we must not forget that MMUs are still nascent to health care delivery and that their operation is often associated with major challenges~\cite{doi:10.1177/0972063417717900}.
For instance, Patro et al.~\cite{Patro08} report long patient waiting times as a result of a high workload, while other studies had problems with small or decreasing number of patient visits~\cite{Schwartze2017,geoffroy2014bringing}.
There is thus a general consensus, that better strategies for the pre-launch of an MMU service are required to improve its effectiveness and sustainability~\cite{doi:10.1177/0972063417717900}.

To contribute to these efforts, we study the combined location and capacity planning for MMU services at the strategic level in form of a capacitated set covering problem called the \textit{strategic planning problem for MMUs} $(\strategicPlanningProblem)$.
The $\strategicPlanningProblem$ addresses the problem of deciding where MMU operation sites should be set up and how often these should be serviced in the course of a week. 
As MMUs are intended as a complementary form of health provision that should be integrated into the present primary care systems~\cite{DOERNER20071078}, we include existing practices with their treatment capacities into our model.
In addition, we consider two fundamentally different types of patient demands that are common in primary care systems:
(i) patients who seek health services via a centralized appointment system and can be steered towards any available treatment facility within the patients' consideration sets and 
(ii) walk-in patients who forgo the appointment system and always visit the treatment facility of their choice -- which we assume to be the closest to the patient.
Given the nature of these patients, we refer the former as the \textit{steerable patient demands} while we call the latter the \textit{unsteerable patient demands}.

The main focus of this paper is the extension of the $\strategicPlanningProblem$ to uncertain patient demands which are intrinsic to the nature of health care needs~\cite{Fone2003}.
To that end, we model both types of patient demands as random variables that we integrate into our models in a robust optimization framework.
Before we discuss previous related work and go into further detail, we stress our contribution to the field.

\paragraph{Contribution}
In this paper, we define the strategic planning problem for MMUs $(\strategicPlanningProblem)$ as a capacitated set covering problem that includes existing practices and two types of patient demands:
steerable demands corresponding to patients who seek health services via a centralized appointment system and unsteerable demands corresponding to walk-ins who always visit their closest available treatment facility.
We present a compact linear IP formulation for the problem which we subsequently solve via Benders decomposition and constraint generation.
To account for uncertainties in both types of patient demands, we introduce exact constraint generation algorithms to solve the robust counterpart of the Benders formulation for interval and budgeted uncertainty sets.
Finally, we show that all our results directly translate to the session-specific version of the problem and evaluate our formulations on a set of test instances based on a rural real-world primary care system in Germany.
Especially the robust formulation based on budgeted uncertainty sets leads to high quality solutions in an acceptable time frame while limiting the price of robustness.
To the best of our knowledge, this is the first contribution that studies the allocation of MMUs as a robust set covering problem.
The concept of modeling two different types of demands has not been considered so far and represents a new extension to the field of location planning.

\paragraph{Outline} We structure the remainder of this paper as follows.
First, we review the related work in Section~\ref{sec:lit} before we formalize the $\strategicPlanningProblem$ in Section~\ref{sec:prob_def} by providing an integer programming formulation which we solve via Benders decomposition.
In Section~\ref{sec:uncertainty}, we extend the problem to uncertainties in both, the steerable and the unsteerable demands.
To overcome the drawbacks of the session-aggregation of patient demands, we introduce a session-specific variant of the $\strategicPlanningProblem$ in Section~\ref{sec:multi_period} and show that all our previous results transfer to this setting.
Finally, we conclude with an extensive computational study based on real-world inspired instances in Section~\ref{sec:computational_study} and a brief discussion in Section~\ref{sec:conclusion}.

%% file: literature.tex
\section{Literature review}
\label{sec:lit}
To the best of our knowledge, one of the earliest works on the planning of MMUs is due to Hodgson et al.~\cite{doi:10.1111/0022-4146.00113} who consider a location-routing problem for a single MMU, i.e., the authors simultaneously determine the vehicle stops and the vehicle route.
In a feasible solution, a set of population centers must be within a prespecified distance of a vehicle stop along the planned route and the objective is to minimize the total length of the route.
In a follow-up article, Hachicha et al.~\cite{HACHICHA200029} extend this problem setting to multiple vehicles and vehicle stops that must be serviced.
All MMU routes must start and end at a central depot and the number of stops per route as well as the total route length is bounded to ensure a balanced workload between MMUs.
Doerner et al.~\cite{DOERNER20071078} consider the problem formulation in~\cite{doi:10.1111/0022-4146.00113} for multiple objective functions.
That is, they evaluate MMU routes with respect to three criteria:
economic efficiency, average access distance, and coverage.
Ozbaygin et al.~\cite{OZBAYGIN2016226} extend the coverage objective for the one vehicle setting to partial coverage, i.e., only the population centers that are visited by an MMU are completely covered while population centers in reach of an MMU stop are only covered with a certain percentage.
More recently, Y\"ucel et al.~\cite{Yuecel18} further generalize the idea of partial coverage to multiple vehicles and integrate their MIP formulation into a data-driven optimization framework based on credit card transactional data.

The main difference between the previous articles and this article is the considered setting: 
While the former focus on very extensive regions with bad road infrastructure and MMU routes that can be multiple weeks long, we consider the problem on a much smaller scale with MMU routes that service at most two stops per day and return to a depot each night. 
As a result, the vehicle routing plays a far more important role in \cite{doi:10.1111/0022-4146.00113,HACHICHA200029,DOERNER20071078,OZBAYGIN2016226,Yuecel18} and is therefore considered at the strategic level while the incorporation of demands and allocation of treatment capacities are considered downstream once the MMU routes are fixed.
We, on the other hand, consider patient demands and the allocation of treatment capacities at the strategic level and shift the vehicle routing into a subsequent problem that boils down to a matching problem for one depot.

A problem originating in humanitarian logistics that is quite related to the $\strategicPlanningProblem$  is studied by Tricoire et al.~\cite{TRICOIRE20121582}.
In this problem, the authors consider the simultaneous setup of distribution centers and the routing of vehicles that restock these with relief goods.
Distribution centers have a certain capacity and their setup induces cost.
The demand for relief goods at the population centers is uncertain and targets the closest distribution center.
The objective is the minimization of the setup and routing costs while the expected coverage of the demands is to be maximized.
The problem is formulated as a stochastic bi-objective combinatorial optimization problem and solved by combining a scenario-based approach with an epsilon-constraint method.
A deterministic single objective variant of this problem that does not consider setup cost for distribution centers and allows demands to be freely assigned among all operated distribution centers within a certain covering distance is studied by Naji-Azimi et al.~\cite{NAJIAZIMI2012596}.

Comparing the $\strategicPlanningProblem$ to the problems in~\cite{TRICOIRE20121582,NAJIAZIMI2012596}, we note that neither of the latter considers existing infrastructure and only either unsteerable or steerable demands but not the combination of the two.
Moreover, both problems put a strong emphasis on the vehicle routing, which we do not consider at all.
Instead, the $\strategicPlanningProblem$ is actually a pure covering location problem, more specifically a set covering problem.
Set covering problems have been studied extensively in various applications and comprehensive review articles on existing work in this field can for example be found in~\cite{caprara2000algorithms,FARAHANI2012368,AHMADIJAVID2017223,garcia2015covering}.
In the following, we will focus our review on set covering problems that incorporate uncertainties in a robust or probabilistic optimization framework.

Probably the most related set covering problem to the $\strategicPlanningProblem$ is the $q$-multiset multicover problem studied by \citet{KRUMKE2019845}.
The $q$-multiset multicover problem is the special case of the decision version of the $\strategicPlanningProblem$ that does not consider existing facilities, setup cost, and unsteerable demands.
The authors study the problem's complexity and  investigate the problem's extension to uncertain demands that may vary within a given interval.
Using budgeted uncertainty sets, they devise a formulation of the robust counterpart which can be solved by constraint generation.
This paper builds on the results in~\cite{KRUMKE2019845} and generalizes them to the $\strategicPlanningProblem$.
Various other studies on robust set covering problems mostly differ in terms of the applied robustness concept.
Dhamdhere et al.~\cite{Dhamdhere2005} introduce demand-robust covering problems and provide approximation algorithms.
Feige et al.~\cite{Feige2007} consider two-stage robust covering problems and devise approximation algorithms, whereas Gupta et al.~\cite{gupta2014thresholded} study approximation algorithms for the $k$-robust set covering problem.
The set covering problem with uncertain cost coefficients is considered by Pereira and Averbakh~\cite{Pereira2013} and exact algorithms for computing min-max regret solutions are presented.

Right-hand side uncertainties in set covering problems in the form of chance constraints, i.e., where demands only have to be covered with a certain probability, are considered under the name probabilistic set covering problem.
Beraldi and Ruszczy\'{n}ski \cite{Beraldi2002} study the probabilistic set covering problem and devise exact methods by enumerating over the set of $p$-efficient points.
Later, Saxena et al.~\cite{Saxena2010} introduce the notion of $p$-inefficiency to devise compact MIP formulations for the probabilistic set covering problem.
Left-hand side uncertainties in set covering problems in the form of chance constraints have been considered under the name uncertain set covering problems.
A polyhedral study of the uncertain set covering problems is performed by Fischetti and Monaci~\cite{Fischetti:cutting-plane-robust} who compare a compact versus a cutting plane model.
More recently, Lutter et al.~\cite{Lutter+etal:cover} introduce compact and non-compact  robust formulations for the uncertain set covering problem by combining concepts from robust and probabilistic optimization.

For more literature on set covering problems under uncertainty, we refer to the references in \cite{Lutter+etal:cover}.
Preference orderings of clients that are similar to our concept of unsteerable patient demands, have been studied for a deterministic facility location problem known as the simple plant location problem in~\cite{HP87} or more recently in~\cite{LQLA07}.

To the best of our knowledge, there are only two previous articles that consider the allocation of MMUs as a set covering problem.
Aguwa et al.~\cite{Aguwa18} focus on data analytics and reduce the MMU allocation to the standard set covering problem.
A more elaborate maximum covering problem for the strategic planning of a single mobile dentistry clinic is considered by Thorsen and McGarvey~\cite{THORSEN20181062} with the goal of improving accessibility while maintaining financial sustainability.
As both of these works consider purely deterministic settings, this paper represents the first contribution to the field that considers the allocation of MMUs as a robust set covering problem.

%% file: prob_def.tex
\section{Problem classification and formulation}
\label{sec:prob_def}
We introduced the strategic planning problem for MMUs $(\strategicPlanningProblem)$ in Section~\ref{sec:introduction}.
In the following, we will formalize the $\strategicPlanningProblem$ and analyze the problem's complexity before we devise solution methods based on mathematical programming techniques. 

The strategic planning problem for MMUs is a capacitated set covering problem that provides the basis for an MMU service:
given a set of potential MMU operation sites~$\Locations$, a set of existing 
primary care practices~$\Practices$, and a set of aggregated patient demand 
origins $\Demands$, decide how many MMU sessions shall be operated at each site 
$\location\in \Locations$ in the course of a week in order to meet all patient 
demands at minimum cost. 
Potential MMU operation sites $\location\in \Locations$ have to be set up at cost 
$\costLocation_\location\in \N$ and allow for up to 
$\capacityLocation_\location\in \N$ operated sessions per week.
Each operated MMU session yields a weekly treatment capacity $\capacitySession\in \N$
and induces the cost $\costSession\in \N$.
Thus, we can define the following.
\begin{definition}
	A \textit{strategic MMU operation plan} is a function $\MMUoperation\colon \Locations \to \N$ that respects the session capacity at each site, i.e., $\MMUoperation_\location\leq \capacityLocation_\location$ for all $\location \in \Locations$ where we notate $\MMUoperation_\location \coloneqq \MMUoperation(\location)$.
	The \textit{cost} of a strategic MMU operation plan $\MMUoperation$ is defined by the costs of setting up sites and operating sessions, i.e., $c(\MMUoperation)\coloneqq \sum_{\location \in \Locations: \MMUoperation_\location>0} \costLocation_\location+ \costSession \, \MMUoperation_\location$.
\end{definition}
Existing primary care practices $\practice\in \Practices$ have an individual 
weekly treatment capacity $\capacityPractice_\practice\in \N$.
Patient demand origins $\demand\in\Demands$ specify the weekly treatment demand 
of a particular region.
To prevent patients from having to travel excessive distances, a \emph{consideration 
	set} $\ConsiderationSet(\demand) \subseteq  \Locations \cup \Practices$ specifies 
for every demand origin $\demand\in \Demands$ the feasible treatment facilities.
The weekly patient demand at each demand origin $\demand \in \Demands$ consists of two 
types of demands:
(i) the steerable demands $\steerableDemand_\demand \in \N$ corresponding to 
patients who announce themselves through a centralized appointment 
system and can be steered to any operating treatment facility in the consideration set 
$\ConsiderationSet(\demand)$ and
(ii) the unsteerable demands $\unsteerableDemand_\demand \in \N$ corresponding to 
walk-in patients that always visit the nearest operating treatment facility
$\closestFacility{\MMUoperation}{\demand} \in \ConsiderationSet(\demand)$
which depends on a given distance measure $\dist\colon \Demands \times \left(\Locations \cup \Practices\right) \to \N$ and the strategic MMU operation plan $\MMUoperation$.

In a feasible strategic MMU operation plan, all steerable patient demands have to be assigned to a feasible treatment facility and every facility's treatment capacity has to be respected.
To formalize these requirements, we first define an assignment of the steerable patient demands to the treatment facilities.

\begin{definition}
	An \textit{assignment} of the steerable patient demands is a set of functions $\left\lbrace \assignmentSteerablePatientDemand{\demand}\right\rbrace _{\demand \in\Demands}$ with $\assignmentSteerablePatientDemand{\demand} \colon \ConsiderationSet(\demand) \to \N$ that distribute all steerable patient demands within their respective consideration set, i.e., $\sum_{k \in \ConsiderationSet(\demand)} \assignmentSteerablePatientDemand{\demand}(k)= \steerableDemand_\demand$ for all $\demand \in \Demands$.
\end{definition}

Next, we define \textit{feasible} strategic MMU operation plans. To ease notation, we denote all patient demand origins that can target a treatment facility $k \in 
\Locations \cup \Practices$ by $\ConsiderationSet(k)\coloneqq 
\{\demand\in \Demands : k \in \ConsiderationSet(\demand) \}$.

\begin{definition}
	A strategic MMU operation plan $\MMUoperation$ is \textit{feasible} if there exists an assignment of the steerable patient demands $\left\lbrace \assignmentSteerablePatientDemand{\demand}\right\rbrace _{\demand \in\Demands}$ that respects the treatment capacity at each treatment facility $k \in \Locations \cup \Practices$; that is 
	\begin{align*}
	\sum_{\demand \in \Demands: \closestFacility{\MMUoperation}{\demand}=k} \unsteerableDemand_\demand + \sum_{\demand \in \ConsiderationSet(k)} \assignmentSteerablePatientDemand{\demand}(k) \leq \begin{cases}
	\capacityPractice_k \quad &\text{if } k\in \Practices,\\
	\capacitySession \, \MMUoperation_k & \text{if } k\in \Locations.
	\end{cases}
	\end{align*}
\end{definition}

Using the notion of a feasible MMU operation plan, we can finally provide a formal definition for the $\strategicPlanningProblem$.

\begin{definition}[$\strategicPlanningProblem$]
	Let the potential MMU operation sites $\location\in \Locations$ with setup costs $\costLocation_\location\in \N$ and weekly session capacities $\capacityLocation_\location \in \N$ be given.
	Moreover, let $\practice\in \Practices$ be the existing practices with weekly treatment capacities $\capacityPractice_\practice \in \N$, and 
	$\demand \in \Demands$ be the patient demand origins with consideration sets $\ConsiderationSet(\demand)\subseteq \Locations \cup \Practices$ and weekly steerable and unsteerable demands $\steerableDemand_\demand,\unsteerableDemand_\demand \in \N$.
	Then, the \textit{strategic planning problem for MMUs} ($\strategicPlanningProblem$) asks for a feasible strategic MMU operation plan of minimum cost, where every operated MMU session induces the cost $\costSession\in \N$ and yields a weekly treatment capacity $\capacitySession\in \N$.
\end{definition}

Classifying the $\strategicPlanningProblem$, we begin by showing that the problem
 is strongly NP-hard.

\begin{theorem}
	The $\strategicPlanningProblem$ is strongly NP-hard.
	\label{NP_SPMMU}
\end{theorem}
\begin{proof}
 By setting $\capacitySession=3$, $\Practices=\emptyset$, $\costSession=1$, $\unsteerableDemand_\demand=0$ for all 
	$\demand\in\Demands$, $\costLocation_\location=0$ for all 
	$\location\in\Locations$ and choosing $\capacityLocation_\location$ large enough, 
	e.g.,~$\capacityLocation_\location=\sum_{\demand\in\Demands}d_v$, it becomes evident that $3$-multiset multicover is a special case of the decision version of the $\strategicPlanningProblem$.	
Thus, the strong NP-hardness result for the $\strategicPlanningProblem$ follows directly from the corresponding result for $3$-multiset multicover in Krumke et al.~\cite{KRUMKE2019845}.

\end{proof}

To solve the $\strategicPlanningProblem$, we present an integer linear programming formulation that we subsequently solve by a Benders decomposition approach.
Let variables
$y_\location \in \{0,1\}$ indicate whether site $\location \in \Locations$ is set up, let variables $x_\location \in \N$ denote the number of weekly operated MMU sessions 
at site $\location \in \Locations$, and let variables $z_{\demand k}\in \N$ determine the 
weekly steerable demand originating in demand origin $\demand \in \Demands$ that is 
assigned to the treatment facility $k \in \ConsiderationSet(\demand)$.
Moreover, let variables $\varUnsteerableDemand_{\demand k} \in \{0,1\}$ indicate the 
closest operating treatment facility $k \in \ConsiderationSet(\demand)$ that is 
targeted by all unsteerable demands originating in $\demand \in \Demands$. 
To that end, let $\Order{\demand} \colon \{1,\dots, 
|\ConsiderationSet(\demand)|\} \to \ConsiderationSet(\demand)$ define an order on the  consideration set  $\ConsiderationSet(\demand)$ that is non-decreasing with respect to the treatment facility's distance $\dist\colon \Demands \times (\Locations \cup \Practices) \to \N$ to demand origin $\demand\in \Demands$.
As a result, $\order{\demand}{1}\in \ConsiderationSet(\demand)$ denotes the closest treatment facility to demand origin $\demand\in \Demands$.
To ease notation, we denote all potential MMU operation sites and practices within the consideration set of demand origin $\demand \in \Demands$ by $\ConsiderationSet_\Locations(\demand) \coloneqq \ConsiderationSet(\demand)\cap \Locations$ and $\ConsiderationSet_\Practices(\demand) \coloneqq \ConsiderationSet(\demand)\cap \Practices$, respectively.
We can now formulate the $\strategicPlanningProblem$ as follows:

\begin{mini!}|s|[2]
	{y,x,z,\varUnsteerableDemand}
	{\sum_{\location\in \Locations} \costLocation_\location \, 
		y_\location + \sum_{\location \in \Locations} \costSession \, 
		x_\location \label{1}}
	{\label{IP}}
	{\OriginalFormulation\quad}
	\addConstraint{x_\location}{\leq \capacityLocation_\location \, y_\location }{\quad \forall \location\in \Locations \label{2}}
	\addConstraint{\sum_{k\in \ConsiderationSet(\demand)} z_{\demand k}}{\geq 
		\steerableDemand_\demand}{\quad \forall \demand\in \Demands\label{3}}
	\addConstraint{\sum_{\demand\in \ConsiderationSet(\location)}  z_{\demand \location} + \sum_{\demand\in \ConsiderationSet(\location)}
		\unsteerableDemand_\demand \, \varUnsteerableDemand_{\demand \location}}{\leq 
		\capacitySession \, x_\location}{\quad \forall \location\in \Locations \label{4}}
	\addConstraint{\sum_{\demand\in \ConsiderationSet(\practice)} z_{\demand \practice} +  \sum_{\demand\in \ConsiderationSet(\practice)}
		\unsteerableDemand_\demand \, \varUnsteerableDemand_{\demand \practice}}{\leq 
		\capacityPractice_\practice}{\quad \forall \practice\in \Practices \label{5}}
	\addConstraint{\sum_{k\in \ConsiderationSet(\demand)} \varUnsteerableDemand_{\demand k}}{\geq 
		1}{\quad \forall \demand\in \Demands \label{6}}
	\addConstraint{\varUnsteerableDemand_{\demand \location}}{\leq y_\location}{\quad \forall \demand 
		\in \Demands,\, \forall \location \in \ConsiderationSet_\Locations(\demand)
		\label{7}}
	\addConstraint{\varUnsteerableDemand_{\demand \location}}{\geq y_\location - 
		\sum_{i=1}^{\invorder{\demand}{\location} -1 } \varUnsteerableDemand_{\demand, 
			\order{\demand}{i}}}{\quad \forall \demand \in \Demands,\, \forall \location \in  
		\ConsiderationSet_\Locations(\demand) \label{8}}
	\addConstraint{\varUnsteerableDemand_{\demand \practice}}{\geq 1 - 
		\sum_{i=1}^{\invorder{\demand}{\practice} -1} \varUnsteerableDemand_{\demand, 
			\order{\demand}{i}}}{\quad \forall \demand \in \Demands,\, \forall \practice \in 
		\ConsiderationSet_\Practices(\demand)   \label{9}}
	\addConstraint{x_\location\in \N,\; y_\location\in \{0,1\} }{}{\quad \forall \location \in 
		\Locations}
	\addConstraint{\varUnsteerableDemand_{\demand k}\in \{0,1\},\;z_{\demand k}\in \N}{}{\quad \forall \demand \in \Demands,\, \forall k\in 
		\ConsiderationSet(\demand).}	
\end{mini!}

In this formulation, constraints \eqref{2} enforce the session capacity at each set up site,
inequalities \eqref{3} model the assignment of the steerable patient demands,
and constraints \eqref{4}--\eqref{5} guarantee that the treatment capacities at each treatment facility are adhered to.
Moreover, inequalities \eqref{6}--\eqref{9} ensure that unsteerable patient demands target their closest considered operating treatment facility.
Showing that $\OriginalFormulation$ is indeed a formulation for the \mbox{$\strategicPlanningProblem$} is rather straight forward but slightly technical.
We therefore only state the result at this point and provide a formal proof in Appendix~\ref{appendix:detcorrect}.

\begin{restatable}{thm2}{DetCor}
\label{thm_first_formulation}
$\OriginalFormulation$ is an integer linear formulation for the $\strategicPlanningProblem$.
\end{restatable}

The $\strategicPlanningProblem$ determines the set up MMU operation sites, the 
number of weekly sessions operated per site, as well as the assignment of the 
steerable patient demands to the treatment facilities.
In the subsequent section, we assume that the actual patient demands are uncertain and reveal themselves only
after we have fixed our decisions regarding the set up sites and operated MMU sessions.
Within this setting, it is no longer expedient to determine one fixed assignment of the steerable patient demands that is feasible for all demand realizations.
Instead, we model a flexible assignment of the steerable demands that can be adjusted once the actual demands are known.

Adding assignment variables for every potential demand realization to $\OriginalFormulation$ leads to a huge model extension that is likely to be computationally intractable.
We therefore propose an alternative formulation for the deterministic $\strategicPlanningProblem$ that considers the steerable patient demands in a subproblem and is thus much better suited to uncertain patient demands.
To that end, we extend the results in~\cite{KRUMKE2019845} and employ a Benders decomposition approach to $\OriginalFormulation$ that  decides and fixes the strategic MMU operation plan in the master problem and only checks the plan's feasibility in the subproblem.  
More precisely, we choose our first stage variables to be $y$, $x$, and $w$ and our second stage variables to be $z$.
The resulting equivalent reformulation of $\OriginalFormulation$ then reads
\begin{mini!}|s|[2]
	{y,x,\varUnsteerableDemand}
	{\sum_{\location\in \Locations} \costLocation_\location \, 
		y_\location + \sum_{\location \in \Locations} \costSession \, 
		x_\location \label{}}
	{}{\MasterProblem\quad}
		\addConstraint{\BendersSubproblem{y}{x}{\varUnsteerableDemand} }{\text{ is feasible} \label{3g}}{}{}
%		\addConstraint{\eqref{2}-\eqref{13}}{\text{ are satisfied}}{}{}		
	\addConstraint{x_\location}{\leq \capacityLocation_\location \, y_\location }{\quad \forall \location\in \Locations \label{}}
	\addConstraint{\sum_{k\in \ConsiderationSet(\demand)} \varUnsteerableDemand_{\demand k}}{\geq 
		1}{\quad \forall \demand\in \Demands \label{}}
	\addConstraint{\varUnsteerableDemand_{\demand \location}}{\leq y_\location}{\quad \forall \demand 
		\in \Demands,\, \forall \location \in \ConsiderationSet_\Locations(\demand) 
		\label{}}
	\addConstraint{\varUnsteerableDemand_{\demand \location}}{\geq y_\location - 
		\sum_{i=1}^{\invorder{\demand}{\location} -1 } \varUnsteerableDemand_{\demand, 
			\order{\demand}{i}}}{\quad \forall \demand \in \Demands,\, \forall \location \in  
		\ConsiderationSet_\Locations(\demand) \label{}}
	\addConstraint{\varUnsteerableDemand_{\demand \practice}}{\geq 1 - 
		\sum_{i=1}^{\invorder{\demand}{\practice} -1} \varUnsteerableDemand_{\demand, 
			\order{\demand}{i}}}{\quad \forall \demand \in \Demands,\, \forall \practice \in 
		\ConsiderationSet_\Practices(\demand) \label{}}
	\addConstraint{x_\location\in \N,\; y_\location\in \{0,1\} }{}{\quad \forall \location \in 
	\Locations}
	\addConstraint{\varUnsteerableDemand_{\demand k}}{\in \{0,1\}}{\quad \forall \demand \in \Demands,\, \forall k\in 
		\ConsiderationSet(\demand), \label{}}
\end{mini!}
where $\BendersSubproblem{\hat{y}}{\hat{x}}{\hat{w}}$ denotes the Benders subproblem for fixed first-stage decisions $\hat{y}$, $\hat{x}$, and $\hat{\varUnsteerableDemand}$, which is defined as
\begin{mini!}|s|[0]
	{z}
	{0}
	{}{\BendersSubproblem{\hat{y}}{\hat{x}}{\hat{w}}\quad}
	\addConstraint{\sum_{k\in \ConsiderationSet(\demand)} z_{\demand k}}{\geq 
		\steerableDemand_\demand}{\quad \forall \demand\in \Demands\label{4b}}
	\addConstraint{\sum_{\demand\in \ConsiderationSet(\location)}  z_{\demand \location} }{\leq 
		\capacitySession \, \hat{x}_\location - \sum_{\demand\in \ConsiderationSet(\location)} 
		\unsteerableDemand_\demand \, \hat{\varUnsteerableDemand}_{\demand \location}}{\quad \forall \location\in \Locations \label{4c}}
	\addConstraint{\sum_{\demand\in \ConsiderationSet(\practice)}z_{\demand \practice} }{\leq 
		\capacityPractice_\practice - \sum_{\demand\in \ConsiderationSet(\practice)}  
		\unsteerableDemand_\demand \, \hat{\varUnsteerableDemand}_{\demand \practice} }{\quad \forall \practice\in \Practices \label{4d}}
	\addConstraint{z_{\demand k}}{\in \N}{\quad \forall \demand \in \Demands,\, \forall k\in 
		\ConsiderationSet(\demand). \label{4e}}
\end{mini!}

Next, we investigate the feasibility of the Benders subproblem $\BendersSubproblem{\hat{y}}{\hat{x}}{\hat{w}}$ to derive Benders feasibility cuts which enforce constraint \eqref{3g}.
Let us first note, that the constraint matrix of 
$\BendersSubproblem{\hat{y}}{\hat{x}}{\hat{w}}$ is totally unimodular.

\begin{lemma}
	\label{lem:TU}
	The constraint matrix of $\BendersSubproblem{\hat{y}}{\hat{x}}{\hat{w}}$ is totally unimodular.
\end{lemma}
 \begin{proof}
 	All entries in the constraint matrix of $\BendersSubproblem{\hat{y}}{\hat{x}}{\hat{w}}$ are in $\{0,1\}$.
 	Moreover, in every column of the constraint matrix at most two entries take the value $1$: one in the rows corresponding to constraints \eqref{4b} and one in the rows corresponding to constraints \eqref{4c} or \eqref{4d}. 
 	Thus, by partitioning the rows of our constraint matrix into the rows corresponding to constraints \eqref{4b}, and the rows that correspond to the constraints \eqref{4c} and \eqref{4d}, the total unimodularity of the constraint matrix of $\BendersSubproblem{\hat{y}}{\hat{x}}{\hat{w}}$ follows directly from the theorem of Hoffman and Gale~\cite{heller1956extension}.
 \end{proof}

As the right hand sides of the constraints in $\BendersSubproblem{\hat{y}}{\hat{x}}{\hat{w}}$ are integral, Lemma~\ref{lem:TU} and Cramer's rule yield that the LP-relaxation of $\BendersSubproblem{\hat{y}}{\hat{x}}{\hat{w}}$ has an integer solution whenever it is feasible.
Thus, we can relax constraints \eqref{4e} and get the following result.
\begin{corollary}
	The Benders subproblem $\BendersSubproblem{\hat{y}}{\hat{x}}{\hat{w}}$ is feasible if and only if its LP-relaxation $\LPBendersSubproblem{\hat{y}}{\hat{x}}{\hat{w}}$ is feasible.
	\label{cor:LPrelax}
\end{corollary}

In order to obtain our Benders feasibility cuts, we exploit the fact that $\LPBendersSubproblem{\hat{y}}{\hat{x}}{\hat{w}}$ is the decision version of a maximum flow problem.
To ease notation, we define the \emph{residual treatment capacity} of a treatment facility in the Benders subproblem as the treatment capacity that remains after the assignment of the unsteerable patient demands is fixed, i.e.,

\begin{align*}
\helper_\location&\coloneqq\capacitySession\, \hat{x}_\location - 
\sum_{\demand\in\ConsiderationSet(\location)}\unsteerableDemand_\demand\, \hat{\varUnsteerableDemand}_{\demand 
	\location} & \forall \location\in\Locations,\\
\helper_\practice&\coloneqq\capacityPractice_\practice 
-\sum_{\demand\in\ConsiderationSet(\practice)}\unsteerableDemand_\demand\,
\hat{\varUnsteerableDemand}_{\demand 
	\practice} & \forall \practice\in\Practices.
\end{align*}
Throughout this paper, we will always assume that the residual treatment capacities are non-negative.

\begin{assumption}
	\label{ass:1}
	For all feasible solutions $(\hat{y},\hat{x},\hat{w})$ to the master problem $\MasterProblem$ without constraint \eqref{3g}, it holds that the residual capacities $\helper_k \geq 0$ for all $k\in \Locations \cup \Practices$.
\end{assumption}

\noindent
Obviously, Assumption~\ref{ass:1} does not hold in general.
However, we can easily enforce Assumption~\ref{ass:1} by adding additional constraints to $\MasterProblem$. 
As this does not offer new insights but only complicates our formulation, we cover the explicit enforcement of Assumption~\ref{ass:1} in Appendix~\ref{appendix:enforcing_assumption}.

The flow network corresponding to $\LPBendersSubproblem{\hat{y}}{\hat{x}}{\hat{w}}$ is now constructed as follows.
Let $G$ be the directed graph with vertex set 
$V(G)=\{s\}\cup\Demands\cup\left(\Locations\cup \Practices\right)\cup  \{t\}$ and 
arc set $E(G)=E_1\cup E_2\cup E_3$, where 
\begin{align*}
E_1&\coloneqq\{(s,\demand)\colon \demand\in\Demands\},\\
E_2&\coloneqq\{(k,t)\colon k\in \Locations\cup \Practices\},\\
E_3&\coloneqq\{(\demand, k)\colon \demand\in \Demands, k \in \ConsiderationSet(\demand)\}.
\end{align*}
We set the capacities of arcs $e_1=(s,\demand)\in E_1$ to $\arcCapacity(e_1)\coloneqq\steerableDemand_\demand$ and the capacities 
of arcs $e_2=(k,t)\in E_2$ to $\arcCapacity(e_2)\coloneqq\helper_k$. The capacities of all arcs $e_3\in E_3$ are set to $\arcCapacity(e_3)\coloneqq\infty$.
Note, that this choice of arc capacities requires Assumption~\ref{ass:1} to hold as we might otherwise end up with negative arc capacities. 
An example of the constructed network $(G,\arcCapacity,s,t)$ can be found in Figure~\ref{fig:network}.
The following now holds true.

\begin{figure}
	\centering
	\includegraphics{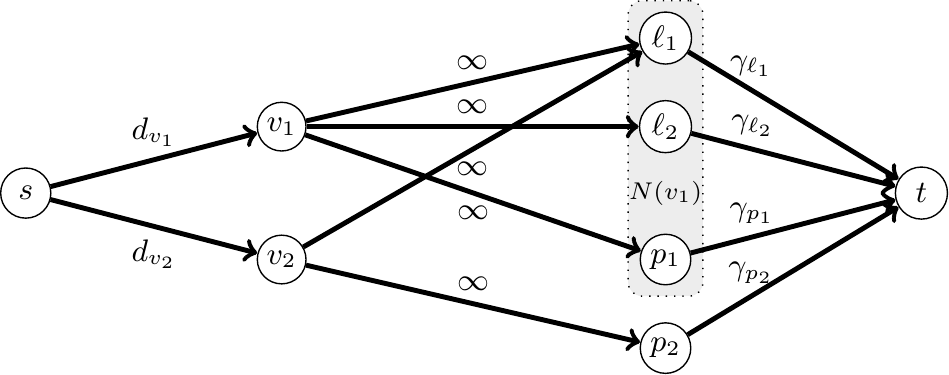}
\caption{Example for the network $(G,\arcCapacity,s,t)$ constructed for Lemma~\ref{BendersEquiv}.}
\label{fig:network}
\end{figure}

\begin{lemma}
	\label{BendersEquiv}
	The Benders subproblem $\BendersSubproblem{\hat{y}}{\hat{x}}{\hat{w}}$ is feasible if and only if the maximum $s$-$t$-flow in the network $(G,\arcCapacity,s,t)$ has a flow value of at least 
	$D:=\sum_{\demand\in\Demands}\steerableDemand_\demand$.
\end{lemma}
\begin{proof}
	Let $f\colon E(G) \to \mathbb{R}_+$ be an $s$-$t$ flow in $(G,\arcCapacity,s,t)$ of value $\text{value}(f) \geq D$.
	We define a solution for $\LPBendersSubproblem{\hat{y}}{\hat{x}}{\hat{w}}$ by setting $z_{\demand k}\coloneqq f((\demand,k))$ for all $\demand \in \Demands$, $k \in \ConsiderationSet(\demand)$ and show that $z$ is feasible.
As the $s$-$t$ cut induced by $S\coloneqq\{s\}$ has capacity $\arcCapacity(\delta^+(S))= D$, it follows that  $\text{value}(f) = D$.
It must thus hold for all arcs $(s,v) \in E_1$ that $f((s,v)) =  \steerableDemand_\demand$ and by flow-conservation we get that for all $v\in V$ 
\begin{align*}
	&\sum_{k\in \ConsiderationSet(\demand)} z_{\demand k} =\sum_{k\in \ConsiderationSet(\demand)} f((\demand,k)) = \sum_{e\in \delta^+(\demand)} f(e) = \sum_{e\in \delta^-(\demand)} f(e) = f((s,\demand)) = \steerableDemand_\demand.
\intertext{
Moreover, for $k \in \Locations \cup \Practices$ we have that 
}
	&\sum_{\demand\in \ConsiderationSet(k)}z_{\demand k} = \sum_{\demand\in \ConsiderationSet(k)} f((\demand,k)) = \sum_{e\in \delta^-(k)} f(e) = \sum_{e\in \delta^+(k)} f(e) = f((k,t)) \leq \arcCapacity((k,t)) = \helper_k.
\end{align*}
As a result, $z$ defines a feasible solution for $\LPBendersSubproblem{\hat{y}}{\hat{x}}{\hat{w}}$ which implies the feasibility of $\BendersSubproblem{\hat{y}}{\hat{x}}{\hat{w}}$ by Corollary~\ref{cor:LPrelax}.
The converse direction can be shown analogously.
\end{proof}

We can now combine our intermediate results to derive Benders feasibility cuts by the application of the max-flow min-cut theorem~\cite{10.5555/137406}.

\begin{theorem}
	The Benders subproblem $\BendersSubproblem{\hat{y}}{\hat{x}}{\hat{w}}$ is feasible if and only if
	\begin{align}
&  \sum_{\demand \in U} \steerableDemand_\demand + \sum_{k\in N(U)} \sum_{\demand \in N(k)} \unsteerableDemand_\demand \,\hat{\varUnsteerableDemand}_{\demand k} \leq 
\sum_{\location \in \ConsiderationSet_\Locations(U)} \capacitySession \,\hat{x}_\location   + \sum_{\practice \in \ConsiderationSet_\Practices(U)} \capacityPractice_\practice \quad  &\forall U\subseteq \Demands,
\label{BendersOptCuts}
\end{align}
 where	$\ConsiderationSet(U)\coloneqq  \bigcup_{\demand \in U} \ConsiderationSet(\demand)$, $\ConsiderationSet_\Locations(U)\coloneqq  \ConsiderationSet(U) \cap \Locations$, and $\ConsiderationSet_\Practices(U)\coloneqq  \ConsiderationSet(U) \cap \Practices$ for $U\subseteq \Demands$.
 \label{thm:bendersoptcuts}
\end{theorem}
\begin{proof}
	First, let us note that \eqref{BendersOptCuts} can be equivalently reformulated as
		\begin{align}
	&  \sum_{\demand \in U} \steerableDemand_\demand + \sum_{k\in N(U)} \sum_{\demand \in N(k)} \unsteerableDemand_\demand \,\hat{\varUnsteerableDemand}_{\demand k} 
	\leq 
	\sum_{\location \in \ConsiderationSet_\Locations(U)} \capacitySession \,\hat{x}_\location   + \sum_{\practice \in \ConsiderationSet_\Practices(U)} \capacityPractice_\practice   &\forall U\subseteq \Demands \notag\\
	\Leftrightarrow & \sum_{\demand \in U} \steerableDemand_\demand 
	\leq \sum_{\location \in \ConsiderationSet_\Locations(U)} \hspace*{-4pt} {\bigg(}  \capacitySession \,\hat{x}_\location \,{-} \hspace*{-3pt} \sum_{\demand \in \ConsiderationSet(\location)} \unsteerableDemand_\demand \,\hat{\varUnsteerableDemand}_{\demand \location}\bigg)   + \hspace*{-2pt} \sum_{\practice \in \ConsiderationSet_\Practices(U)} \hspace*{-4pt} {\bigg(} \capacityPractice_\practice \,{-} \hspace*{-3pt} \sum_{\demand \in \ConsiderationSet(\practice)} \unsteerableDemand_\demand \,\hat{\varUnsteerableDemand}_{\demand \practice} \bigg)  &\forall U\subseteq \Demands \notag\\
	\Leftrightarrow & \sum_{\demand \in U} \steerableDemand_\demand  \leq \sum_{k \in \ConsiderationSet(U)} \helper_k&\forall U\subseteq \Demands. \label{reformulationBendersOptCuts}
	\end{align}
	Moreover, by Lemma~\ref{BendersEquiv} and the max-flow-min-cut theorem, $\BendersSubproblem{\hat{y}}{\hat{x}}{\hat{w}}$ is feasible if and only if every $s$-$t$ cut in the network $(G,\arcCapacity,s,t)$ induced by $S\subsetneq V(G)$ with $s\in S$, $t\notin S$ has capacity $\arcCapacity(\delta^+(S)) \geq D$.
	Hence, it suffices to show that inequalities \eqref{reformulationBendersOptCuts} hold if and only if every $s$-$t$ cut induced by $S\subsetneq V(G)$ has capacity $\arcCapacity(\delta^+(S)) \geq D$.
	
	Assume that the inequalities \eqref{reformulationBendersOptCuts} hold.
	All $s$-$t$ cuts $\delta^+(S)\subseteq E(G)$ containing arcs from $E_3$ have infinite capacity and obviously satisfy $\arcCapacity(\delta^+(S)) \geq D$.
	Hence, let $\delta^+(S)$ with $s\in S$, $t\notin S$ be an $s$-$t$ cut in $G$ of finite capacity and define $U \coloneqq S \cap V$.
	Then $N(U) \subseteq S$ as otherwise $\delta^+(S) \cap E_3 \neq \emptyset$. Consequently we have that
	\begin{align*}
		\arcCapacity(\delta^+(S)) \geq \sum_{k \in \ConsiderationSet(U)} \helper_k + \sum_{\demand \in \Demands \setminus U} \steerableDemand_\demand \geq   \sum_{\demand \in U} \steerableDemand_\demand + \sum_{\demand \in \Demands \setminus U} \steerableDemand_\demand = D.
	\end{align*}
	
	Conversely, assume that  $\arcCapacity(\delta^+(S)) \geq D$ for all $S\subsetneq V(G)$ with $s\in S$, $t\notin S$.
	Let $U \subseteq \Demands$ and define $S \coloneqq \{s\} \cup U \cup N(U)$.
	Then obviously $s \in S$ and $t\notin S$ and we get by our assumption that
	\begin{align*}
		\arcCapacity(\delta^+(S)) =&\sum_{k\in N(U)} \helper_k + \sum_{\demand \in \Demands \setminus U} \steerableDemand_\demand \geq  \sum_{\demand \in \Demands} \steerableDemand_\demand \\
		\Leftrightarrow &\sum_{k\in N(U)} \helper_k \geq \sum_{\demand \in U} \steerableDemand_\demand.	
	\end{align*} 
\end{proof}
As a result of Theorem~\ref{thm:bendersoptcuts}, we can obtain a linear formulation of our Benders master problem $\MasterProblem$  by substituting constraint \eqref{3g} with the Benders feasibility cuts \eqref{BendersOptCuts}.
We refer to the resulting formulation of the $\strategicPlanningProblem$ as $\BendersFormulation$.

\begin{corollary}
	$\BendersFormulation$ is an integer linear formulation for the $\strategicPlanningProblem$.
\end{corollary}

The Benders feasibility cuts \eqref{BendersOptCuts} can be separated in 
polynomial time by computing a minimum $s$-$t$ cut in the network $(G,\arcCapacity,s,t)$ 
as described above.
Alternatively, one can separate the cuts by solving the dual of the Benders subproblem $\LPBendersSubproblem{\hat{y}}{\hat{x}}{\hat{w}}$ that we describe in Appendix~\ref{appendix:separationLP}.
Appendix~\ref{appendix:separationLP} furthermore shows that the separation problem for $\BendersFormulation$ is trivial if we only consider unsteerable patient demands due to Assumption~\ref{ass:1}.

The next section considers the $\strategicPlanningProblem$ with uncertain patient demands. 
As our main interest lies on the setup of operation sites and operation of MMU sessions,
it suffices to guarantee the existence of a feasible assignment of the steerable patient 
demands.
Thus, we restrict ourselves to the Benders formulation $\BendersFormulation$ of the $\strategicPlanningProblem$ in the following and note that such an assignment can be determined by a single maximum flow computation as a result of Lemma~\ref{BendersEquiv}.

%% file: uncertainty_general.tex
\section{Integration of demand uncertainties}
\label{sec:uncertainty}
Up to this point, we considered the $\strategicPlanningProblem$ in a deterministic setting.
That is we assume that all input data is precisely known, in particular, we assume that the weekly steerable and unsteerable patient demand at each demand origin $\demand\in \Demands$ can be described by deterministic nominal values $\steerableDemand_\demand \in \N $ and $\unsteerableDemand_\demand \in  \N$, respectively.
Clearly, this assumption does not hold in reality as a patient's need to see a primary care physician is subject to uncertainty.
As a result, strategic MMU operation plans that are feasible with respect to the nominal patient demands may be infeasible in real-life operation~\cite{BenTal+Nemirovsky:2000}.
To address this issue, we model the weekly patient demands at each demand origin as random variables.
Specifically, we assume that the steerable patient demand at each demand origin $\demand \in \Demands$ can be described by an independent random variable $\scenSteerable_\demand$ that takes values in $\{\lowerSteerable_\demand, \lowerSteerable_\demand+1, \dots, \upperSteerable_\demand\}$, where $\lowerSteerable_\demand, \upperSteerable_\demand \in \N$ with $\lowerSteerable_\demand \leq  \upperSteerable_\demand$.
Analogously, we assume that the unsteerable patient demand at each demand origin $\demand \in \Demands$ can be described by an independent random variable $\scenUnsteerable_\demand$ that takes values in $\{\lowerUnsteerable_\demand, \lowerUnsteerable_\demand+1, \dots, \upperUnsteerable_\demand\}$, where $\lowerUnsteerable_\demand, \upperUnsteerable_\demand \in \N$ with $\lowerUnsteerable_\demand \leq  \upperUnsteerable_\demand$.

To extend the $\strategicPlanningProblem$ to uncertain patient demands, we employ the concept of robust optimization~\cite{ben2009robust,GABREL2014471}.
The core principle of robust optimization is to strive for solutions that are, to some extent, immune to variations in the input data.
This is achieved by hedging solutions against a subset of all possible realizations of the uncertain parameters which are represented by so-called \textit{uncertainty sets}.

Under our model of data uncertainty, the set of all possible realizations of the steerable and unsteerable patient demands are given by $\AllScenSteerable\coloneqq \{\scenSteerable \in \N^{\Demands} \colon
\lowerSteerable_\demand \leq \scenSteerable_\demand \leq \upperSteerable_\demand \;\forall \demand \in \Demands\}$ and $\AllScenUnsteerable\coloneqq \{\scenUnsteerable \in \N^{\Demands} \colon 
\lowerUnsteerable_\demand \leq \scenUnsteerable_\demand \leq \upperUnsteerable_\demand \;\forall \demand \in \Demands\}$, respectively.
The \textit{robust strategic planning problem for MMUs} then asks for a strategic MMU operation plan of minimum cost that is feasible for every pair of patient demand realizations $(\scenSteerable, \scenUnsteerable) \in \Uncertainty_\indexSteerable \times \Uncertainty_\indexUnsteerable$, where $\Uncertainty_\indexSteerable \subseteq \AllScenSteerable$ is an uncertainty set of the steerable patient demands and $\Uncertainty_\indexUnsteerable \subseteq \AllScenUnsteerable$ is an uncertainty set of the unsteerable patient demands.
 To formalize this, we extend the notion of a feasible strategic MMU operation plan to the robust setting with uncertain patient demands.

 \begin{definition}
	A strategic MMU operation plan $\MMUoperation$ is \textit{robust feasible}
	if $\MMUoperation$ is feasible for the deterministic $\strategicPlanningProblem$ with nominal patient demands $\steerableDemand= \scenSteerable$ and $ \unsteerableDemand= \scenUnsteerable$ for every pair of patient demand realizations $(\scenSteerable, \scenUnsteerable) \in \Uncertainty_\indexSteerable \times \Uncertainty_\indexUnsteerable$.
\end{definition}

Note, that the number of sessions operated at each site is fixed in a robust feasible MMU operation plan which induces a fixed assignment of the unsteerable demands that is independent of the realization $\scenUnsteerable \in \Uncertainty_\indexUnsteerable$.
However, the assignment of the steerable patient demands is variable and can be adapted for any realization $\scenSteerable \in \Uncertainty_\indexSteerable$.
We can now use the notion of a robust feasible strategic MMU operation plan to provide a formal definition of the robust strategic planning problem for MMUs.
\begin{definition}[$\robustStrategicPlanningProblem$]
	Let the potential MMU operation sites $\location\in \Locations$ with setup costs $\costLocation_\location\in \N$ and weekly session capacities $\capacityLocation_\location \in \N$ be given.
	Moreover, let $\practice\in \Practices$ be the existing practices with weekly treatment capacities $\capacityPractice_\practice\in \N$ and 
	$\demand \in \Demands$ be the patient demand origins with consideration sets $\ConsiderationSet(\demand)\subseteq \Locations \cup \Practices$.
	The uncertain weekly steerable and unsteerable demands are described by the uncertainty sets $\Uncertainty_\indexSteerable \subseteq \AllScenSteerable$ and $\Uncertainty_\indexUnsteerable \subseteq \AllScenUnsteerable$, respectively.
	Then, the \textit{robust strategic planning problem for MMUs} ($\robustStrategicPlanningProblem$) asks for a robust feasible strategic MMU operation plan of minimum cost, where every operated MMU session induces the cost $\costSession\in \N$ and yields a weekly treatment capacity $\capacitySession\in \N$.
\end{definition}

Obviously, the $\robustStrategicPlanningProblem$ is a generalization of the $\strategicPlanningProblem$ and thus the problem's strong NP-hardness follows immediately from Theorem~\ref{NP_SPMMU}.

\begin{theorem}
	\label{thm: robust NP hard}
	The $\robustStrategicPlanningProblem$ is strongly NP-hard.
\end{theorem}

To obtain a formulation for the $\robustStrategicPlanningProblem$, we consider the robust counterpart of the formulation $\BendersFormulation$ for the deterministic $\strategicPlanningProblem$ defined as

\begin{mini!}|s|[2]
	{y,x,\varUnsteerableDemand}
	{\sum_{\location\in \Locations} \costLocation_\location \, 
		y_\location + \sum_{\location \in \Locations} \costSession \, 
		x_\location \label{}}
	{}{\RobustBendersFormulation\;\;\,}
	\addConstraint{\max_{\scenSteerable \in \Uncertainty_\indexSteerable} \sum_{\demand \in U} \scenSteerable_\demand {+} \max_{\scenUnsteerable \in \Uncertainty_\indexUnsteerable} \sum_{k\in N(U)} \sum_{\demand \in N(k)} \scenUnsteerable_\demand \,\varUnsteerableDemand_{\demand k} \notag}{}{}
	\addConstraint{\hspace*{1.3cm} \leq
	\sum_{\location \in \ConsiderationSet_\Locations(U)} \capacitySession \,x_\location   + \sum_{\practice \in \ConsiderationSet_\Practices(U)} \capacityPractice_\practice}{}{\quad \forall U\subseteq \Demands \label{robBernderOptCut}}
	% supress repetition of s.t.	
	\togglefalse{bodyCon}
	\addConstraint{x_\location}{\leq \capacityLocation_\location \, y_\location }{\quad \forall \location\in \Locations \label{10c}}
	\addConstraint{\sum_{k\in \ConsiderationSet(\demand)} \varUnsteerableDemand_{\demand k}}{\geq 
		1}{\quad \forall \demand\in \Demands \label{10d}}
	\addConstraint{\varUnsteerableDemand_{\demand \location}}{\leq y_\location}{\quad \forall \demand 
		\in \Demands,\, \forall \location \in \ConsiderationSet_\Locations(\demand) 
		\label{}}
	\addConstraint{\varUnsteerableDemand_{\demand \location}}{\geq y_\location - 
		\sum_{i=1}^{\invorder{\demand}{\location} -1 } \varUnsteerableDemand_{\demand, 
			\order{\demand}{i}}}{\quad \forall \demand \in \Demands,\, \forall \location \in  
		\ConsiderationSet_\Locations(\demand) \label{}}
	\addConstraint{\varUnsteerableDemand_{\demand \practice}}{\geq 1 - 		\sum_{i=1}^{\invorder{\demand}{\practice} -1} \varUnsteerableDemand_{\demand, 
			\order{\demand}{i}}}{\quad \forall \demand \in \Demands,\, \forall \practice \in 
		\ConsiderationSet_\Practices(\demand) \label{10g}}
	\addConstraint{x_\location\in \N,\; y_\location\in \{0,1\} }{}{\quad \forall \location \in 
		\Locations}
	\addConstraint{\varUnsteerableDemand_{\demand k}}{\in \{0,1\}}{\quad \forall \demand \in \Demands,\, \forall k\in 
		\ConsiderationSet(\demand). \label{10i}}
\end{mini!}

In this formulation, inequalities \eqref{robBernderOptCut} correspond to the robust Benders feasibility cuts, constraints \eqref{10c} enforce the session capacity at each setup site, and inequalities \eqref{10d}--\eqref{10g} ensure that unsteerable patient demands target their closest considered operating treatment facility.
We show that $\RobustBendersFormulation$ is a formulation for the $\robustStrategicPlanningProblem$.
\begin{theorem}
	$\RobustBendersFormulation$ is an integer formulation for the $\robustStrategicPlanningProblem$.
\end{theorem}
\begin{proof}
Given an optimal solution $(y,x,w)$ to $\RobustBendersFormulation$, a strategic MMU operation plan $\MMUoperation$ of minimum cost can be defined via $\MMUoperation_\location \coloneqq x_\location$ for all $\location \in \Locations$.
As $(y,x,w)$ satisfies constraints \eqref{robBernderOptCut}, it follows that for every pair of patient demand realizations $(\hat{\scenSteerable}, \hat{\scenUnsteerable}) \in \Uncertainty_\indexSteerable \times \Uncertainty_\indexUnsteerable$ and every $U\subseteq \Demands$ we have that
\begin{align*}
	\sum_{\demand \in U} \hat{\scenSteerable}_\demand + \sum_{k\in N(U)} \sum_{\demand \in N(k)} \hat{\scenUnsteerable}_\demand \,\varUnsteerableDemand_{\demand k} &\leq \max_{\scenSteerable \in \Uncertainty_\indexSteerable} \sum_{\demand \in U} \scenSteerable_\demand + \max_{\scenUnsteerable \in \Uncertainty_\indexUnsteerable} \sum_{k\in N(U)} \sum_{\demand \in N(k)} \scenUnsteerable_\demand \,\varUnsteerableDemand_{\demand k}\\
	&\leq
	\sum_{\location \in \ConsiderationSet_\Locations(U)} \capacitySession \,x_\location   + \sum_{\practice \in \ConsiderationSet_\Practices(U)} \capacityPractice_\practice.
\end{align*}
Thus, the robust feasibility of $\MMUoperation$ follows directly from Theorem~\ref{thm:bendersoptcuts}.
\end{proof}

Formulation $\RobustBendersFormulation$ is in general non-linear due to constraints \eqref{robBernderOptCut}.
However, for certain choices of the uncertainty sets $\Uncertainty_\indexSteerable \subseteq \AllScenSteerable$ and $\Uncertainty_\indexUnsteerable \subseteq \AllScenUnsteerable$ we can show that \eqref{robBernderOptCut} can be reformulated in a linear way.
There are various concepts of defining uncertainty sets; see, e.g.,~\cite{Bertsimas2003, BertsimasPriceOfRobustness2004, kouvelis1996robust, kasperski2008discrete, Kasperski+Zielinski:2016}.
The first of these setting we consider, is the complete protection against uncertainties in the patient demands, i.e., the $\robustStrategicPlanningProblem$ with uncertainty sets $\Uncertainty_\indexSteerable = \AllScenSteerable$ and $\Uncertainty_\indexUnsteerable = \AllScenUnsteerable$. This setting is known as \emph{interval uncertainty}~\cite{ben2009robust,doi:10.1287/opre.21.5.1154} and allows us to reformulate \eqref{robBernderOptCut} as
\begin{align}
	\sum_{\demand \in U} \upperSteerable_\demand + \sum_{k\in N(U)} \sum_{\demand \in N(k)} \upperUnsteerable_\demand \,\varUnsteerableDemand_{\demand k} \leq 
	\sum_{\location \in \ConsiderationSet_\Locations(U)} \capacitySession \,x_\location   + \sum_{\practice \in \ConsiderationSet_\Practices(U)} \capacityPractice_\practice \quad  &\forall U\subseteq \Demands. \tag{$\ref{robBernderOptCut}^\prime$}
\end{align}
That is, we can reduce the $\robustStrategicPlanningProblem$ for this particular choice of uncertainty sets to the deterministic $\strategicPlanningProblem$ with worst-case nominal patient demands $\steerableDemand_\demand = \upperSteerable_\demand$ and $\unsteerableDemand_\demand = \upperUnsteerable_\demand$ for all $\demand \in \Demands$.
We refer to the resulting formulation of the $\robustStrategicPlanningProblem$ with interval uncertainty sets as $\IntervalRobustBendersFormulation$.
This approach is known as the method of Soyster~\cite{doi:10.1287/opre.21.5.1154} and generally entails prohibitive operation cost as a result of the method's conservatism.

To alleviate this drawback, Bertsimas and Sim~\cite{BertsimasPriceOfRobustness2004} introduced \textit{budgeted uncertainty sets} that restrict the deviations in the uncertain input data through a budget parameter.
The choice of this budget parameter allows for a trade-off between robustness and operation cost of the obtained solutions. 
In the following, we consider a slight adaptation of budgeted uncertainty sets that contains all patient demand realizations in which the total (un-)steerable patient demand is bounded by the parameter $\demandBound_\indexSteerable\in \N$ $(\demandBound_\indexUnsteerable\in \N)$.
For the steerable patient demands, these realizations can be represented by the uncertainty set
\begin{align*}
\BudgetedUncertainty_\indexSteerable\coloneqq \left\{\scenSteerable \in \N^{\Demands} \colon 
\lowerSteerable_\demand \leq \scenSteerable_\demand \leq \upperSteerable_\demand \;\forall \demand \in \Demands, 
\;\sum_{\demand \in \Demands} \scenSteerable_\demand \leq \demandBound_\indexSteerable \right\}.
\end{align*}
For the unsteerable patient demands, we analogously obtain the uncertainty set
\begin{align*}
\BudgetedUncertainty_\indexUnsteerable\coloneqq \left\{\scenUnsteerable \in \N^{\Demands} \colon \lowerUnsteerable_\demand \leq \scenUnsteerable_\demand \leq \upperUnsteerable_\demand\; \forall \demand \in \Demands,\; \sum_{\demand \in \Demands} \scenUnsteerable_\demand \leq \demandBound_\indexUnsteerable \right\}.
\end{align*}
To ensure that the uncertainty sets~$\BudgetedUncertainty_\indexSteerable$ and $\BudgetedUncertainty_\indexUnsteerable$ are non-empty, we require that $\sum_{\demand \in \Demands} \lowerSteerable_\demand \leq \demandBound_\indexSteerable$ and $\sum_{\demand \in \Demands} \lowerUnsteerable_\demand \leq \demandBound_\indexUnsteerable$.
Moreover, we can assume w.l.o.g.~that $\demandBound_\indexSteerable \leq \sum_{\demand \in \Demands} \upperSteerable_\demand$ and $\demandBound_\indexUnsteerable \leq \sum_{\demand \in \Demands} \upperUnsteerable_\demand$ as we otherwise always have $\BudgetedUncertainty_\indexSteerable=\AllScenSteerable$ and $\BudgetedUncertainty_\indexUnsteerable= \AllScenUnsteerable$.

For the remainder of this section, we consider the $\robustStrategicPlanningProblem$ with the budgeted uncertainty sets $\BudgetedUncertainty_\indexSteerable$ and $\BudgetedUncertainty_\indexUnsteerable$ and devise an integer linear formulation, which is subsequently solved by constraint generation.
To that end, we show that \eqref{robBernderOptCut} can be linearized for this particular choice of uncertainty sets.

Considering the non-linear part in \eqref{robBernderOptCut} corresponding to the steerable patient demands,
the linear reformulation is straight forward as 
\begin{align}
	\max_{\scenSteerable \in \BudgetedUncertainty_\indexSteerable} \sum_{\demand \in U} \scenSteerable_\demand = 
	\min \left\lbrace \sum_{\demand \in U} \upperSteerable_\demand,\; 
	\demandBound_\indexSteerable- \sum_{\demand \in \Demands \setminus U} \lowerSteerable_\demand \right\rbrace 
\label{linearReformulationSteerableDemand}
\end{align}
which is simply a constant for fixed $U \subseteq \Demands$.

For the non-linear part in \eqref{robBernderOptCut} corresponding to the unsteerable patient demands, we can obtain a linear reformulation through LP duality.
By the definition of $\BudgetedUncertainty_\indexUnsteerable$, we can formulate the inner maximization problem
\begin{align*}
\max_{\scenUnsteerable \in \BudgetedUncertainty_\indexUnsteerable} \sum_{k\in N(U)} \sum_{\demand \in N(k)} \scenUnsteerable_\demand \,\hat{\varUnsteerableDemand}_{\demand k}
\end{align*}
for fixed $U\subseteq \Demands$ and fixed assignment of the unsteerable demands $\hat{\varUnsteerableDemand}_{\demand k} \in \{0,1\}$ for all $\demand\in \Demands$ and $k \in \ConsiderationSet(\demand)$  via the following integer linear program:
\begin{maxi!}|s|[2]
	{\scenUnsteerable}
	{\sum_{k\in N(U)} \sum_{\demand \in N(k)} \scenUnsteerable_\demand \,\hat{\varUnsteerableDemand}_{\demand k}}
	{}
	{\PrimalUnsteerable{\hat{\varUnsteerableDemand}}\quad}
	\addConstraint{\scenUnsteerable_\demand}{\leq \upperUnsteerable_\demand}{\quad \forall \demand \in \Demands \label{12b}}
	\addConstraint{-\scenUnsteerable_\demand}{\leq -\lowerUnsteerable_\demand}{\quad \forall \demand \in \Demands \label{12c}}
	\addConstraint{\sum_{\demand \in \Demands} \scenUnsteerable_\demand}{\leq \demandBound_\indexUnsteerable}{ \label{12d}}
	\addConstraint{\scenUnsteerable_\demand}{\in \N}{\quad\forall \demand \in \Demands. \label{12e}}
\end{maxi!}

The problem $\PrimalUnsteerable{\hat{\varUnsteerableDemand}}$ is feasible and bounded as we assumed $ \demandBound_\indexUnsteerable \geq \sum_{\demand \in \Demands} \lowerUnsteerable_\demand $.
Moreover, we can show that the constraint matrix of $\PrimalUnsteerable{\hat{\varUnsteerableDemand}}$ is totally unimodular.

\begin{lemma}
	\label{lem:TU2}
	The constraint matrix of $\PrimalUnsteerable{\hat{\varUnsteerableDemand}}$ is totally unimodular.
\end{lemma}
\begin{proof}
	The unit rows of the constraint matrix corresponding to constraints \eqref{12b} and \eqref{12c} are irrelevant to the total unimodularity and do not have to be considered~\cite{doi:10.1002/9781118627372.ch14}. Thus, we end up with a vector of ones corresponding to constraint~\eqref{12d} which is obviously totally unimodular as each square submatrix has determinant one. 
\end{proof}
By our choice of parameters, the right hand sides of the constraints in $\PrimalUnsteerable{\hat{\varUnsteerableDemand}}$ are integral.
Thus, the polyhedron of the LP-relaxation $\LPPrimalUnsteerable{\hat{\varUnsteerableDemand}}$ is integral and we can relax the integrality constraint \eqref{12e} as the optimal solution values of $\LPPrimalUnsteerable{\hat{\varUnsteerableDemand}}$ and $\PrimalUnsteerable{\hat{\varUnsteerableDemand}}$ coincide.
The dual problem of $\LPPrimalUnsteerable{\hat{\varUnsteerableDemand}}$ is given by

\begin{mini*}|s|[2]
	{\dualBudgetCap, \kappa, \rho}
	{\sum_{\demand \in \Demands} \left(\upperUnsteerable_\demand \dualBudgetCap_\demand - \lowerUnsteerable_\demand \kappa_\demand \right) + \demandBound_\indexUnsteerable \rho}
	{}
	{\DualUnsteerable{\hat{\varUnsteerableDemand}}\quad}
	\addConstraint{\dualBudgetCap_\demand - \kappa_\demand + \rho}{\geq \sum_{k \in \ConsiderationSet(U)\cap \ConsiderationSet(\demand)} \hat{\varUnsteerableDemand}_{\demand k}}{\quad \forall \demand \in \Demands}
	\addConstraint{\dualBudgetCap_\demand, \kappa_\demand, \rho}{\geq 0}{ \quad \forall \demand \in \Demands.}
\end{mini*}

Strong duality states that the optimal solution values of $\LPPrimalUnsteerable{\hat{\varUnsteerableDemand}}$ and  $\DualUnsteerable{\hat{\varUnsteerableDemand}}$ coincide.
Hence, every feasible solution of $\DualUnsteerable{\hat{\varUnsteerableDemand}}$ yields an upper bound on the optimal solution value of $\PrimalUnsteerable{\hat{\varUnsteerableDemand}}$.
Combined with the observations in~\cite{BertsimasPriceOfRobustness2004}, we can now reformulate \eqref{robBernderOptCut} for the budgeted uncertainty sets $\BudgetedUncertainty_\indexSteerable$ and $\BudgetedUncertainty_\indexUnsteerable$ via the following set of constraints:
\begin{flalign*}
	\;\,\min \left\lbrace \sum_{\demand \in U} \upperSteerable_\demand,\; 
	\demandBound_\indexSteerable- \sum_{\demand \in \Demands \setminus U} \lowerSteerable_\demand \right\rbrace + 
	\sum_{\demand \in \Demands} \left(\upperUnsteerable_\demand \dualBudgetCap^U_\demand - \lowerUnsteerable_\demand \kappa^U_\demand \right) + \demandBound_\indexUnsteerable \rho^U &&\notag
\end{flalign*}
\vspace*{-.4cm}
\begin{align}
&\hspace*{1.7cm}\leq \sum_{\location\in \ConsiderationSet_\Locations(U)} \capacitySession x_\location  + \sum_{\practice\in\ConsiderationSet_\Practices(U)} \capacityPractice_\practice \hspace*{3.121cm} && \forall U \subseteq \Demands \label{reform1} \\
&\dualBudgetCap^U_\demand - \kappa^U_\demand + \rho^U  \geq \sum_{k \in \ConsiderationSet(U) \cap \ConsiderationSet(\demand)} \varUnsteerableDemand_{\demand k}  &&\forall \demand \in \Demands,\, \forall U \subseteq \Demands \label{reform2}\\
&\dualBudgetCap^U_\demand, \kappa^U_\demand, \rho^U \geq 0  &&  \forall \demand \in \Demands,\, \forall U \subseteq \Demands. \label{reform3}
\end{align}
We refer to the resulting formulation of the $\robustStrategicPlanningProblem$ with budgeted uncertainty sets as $\GammaRobustBendersFormulation$. 
Formulation $\GammaRobustBendersFormulation$ is an integer linear program with an exponential number of constraints.
To solve it, we apply constraint generation, i.e., we consider $\GammaRobustBendersFormulation$ with a subset of the constraints of type \eqref{reform1}--\eqref{reform3}.
In particular, we decide on some $\SubsetOfSubsets \subseteq 2^\Demands$ and consider the constraints of type \eqref{reform1}--\eqref{reform3} only for the subsets of patient demand origins $U\in \SubsetOfSubsets$.
This yields a relaxation of $\GammaRobustBendersFormulation$ called the \textit{restricted master problem}.

Once an optimal solution $(\hat{y}, \hat{x}, \hat{\varUnsteerableDemand}, \hat{\dualBudgetCap}, \hat{\kappa}, \hat{\rho})$ to the restricted master problem induced by $\SubsetOfSubsets$ is known, we need to decide whether $(\hat{y}, \hat{x}, \hat{\varUnsteerableDemand}, \hat{\dualBudgetCap}, \hat{\kappa}, \hat{\rho})$ is feasible for the original formulation $\GammaRobustBendersFormulation$.
To that end, we examine whether there exists a subset $U\subseteq \Demands$ for which the system \eqref{reform1}--\eqref{reform3} is infeasible.
This problem is known as the \textit{separation problem} and can be formalized as follows: Is there a subset $U \subseteq \Demands$ such that the system
\begin{align*}
&\min \left\lbrace \sum_{\demand \in U} \upperSteerable_\demand,\; 
\demandBound_\indexSteerable- \sum_{\demand \in \Demands \setminus U} \lowerSteerable_\demand \right\rbrace + 
\sum_{\demand \in \Demands} \left(\upperUnsteerable_\demand \dualBudgetCap^U_\demand - \lowerUnsteerable_\demand \kappa^U_\demand \right) + \demandBound_\indexUnsteerable \rho^U \notag\\
&\hspace*{1.7cm}\leq \sum_{\location\in \ConsiderationSet_\Locations(U)} \capacitySession \hat{x}_\location  + \sum_{\practice\in\ConsiderationSet_\Practices(U)} \capacityPractice_\practice \\
&\dualBudgetCap^U_\demand - \kappa^U_\demand + \rho^U \geq \sum_{k \in \ConsiderationSet(U) \cap \ConsiderationSet(\demand)} \hat{\varUnsteerableDemand}_{\demand k}  &&\forall \demand \in \Demands \\
&\dualBudgetCap^U_\demand, \kappa^U_\demand, \rho^U \geq 0  &&  \forall \demand \in \Demands
\end{align*}
has no solution $(\dualBudgetCap^U, \kappa^U,\rho^U)$, i.e., is infeasible? 

By duality and Farkas' lemma~\cite{doi:10.1002/9781118627372.ch2}, we can equivalently reformulate the separation problem in terms of the original constraints \eqref{robBernderOptCut}: Is there a subset $U \subseteq \Demands$ such that

\begin{align}
	\max_{\scenSteerable \in \BudgetedUncertainty_\indexSteerable} \sum_{\demand \in U} \scenSteerable_\demand + \max_{\scenUnsteerable \in \BudgetedUncertainty_\indexUnsteerable} \sum_{k\in N(U)} \sum_{\demand \in N(k)} \scenUnsteerable_\demand \,\hat{\varUnsteerableDemand}_{\demand k} > \sum_{\location \in \ConsiderationSet_\Locations(U)} \capacitySession \,\hat{x}_\location   + \sum_{\practice \in \ConsiderationSet_\Practices(U)} \capacityPractice_\practice \;? \label{SepOriginal}
\end{align}

In the following, we simplify formulation~\eqref{SepOriginal} of the separation problem even further.
To that end, let us recall that for fixed set $U \subseteq \Demands$ we have concluded in  \eqref{linearReformulationSteerableDemand} that for the steerable patient demands holds
\begin{align}
\max_{\scenSteerable \in \BudgetedUncertainty_\indexSteerable} \sum_{\demand \in U} \scenSteerable_\demand = 
\min \left\lbrace \sum_{\demand \in U} \upperSteerable_\demand,\; 
\demandBound_\indexSteerable- \sum_{\demand \in \Demands \setminus U} \lowerSteerable_\demand \right\rbrace. \tag{\ref{linearReformulationSteerableDemand}}
\end{align}
Moreover, as the assignment of the unsteerable demands in the separation problem is fixed, we can obtain an analogous result for the unsteerable patient demands. 
To that end, let $\closestFixedFacility{\hat{\varUnsteerableDemand}}{\demand}\in \ConsiderationSet(\demand)$ denote the unique treatment facility that is targeted by all unsteerable patient demand originating in $\demand\in \Demands$, i.e., $\hat{\varUnsteerableDemand}_{\demand k} =1$ if and only if $k=\closestFixedFacility{\hat{\varUnsteerableDemand}}{\demand}$.
Moreover, let $\InducedNeighborhood \coloneqq \left\{\demand \in \Demands \colon \closestFixedFacility{\hat{\varUnsteerableDemand}}{\demand}\in \ConsiderationSet(U) \right\}$
denote all demand origins whose unsteerable patient demands target a treatment facility in $\ConsiderationSet(U)\subseteq \Locations \cup \Practices$. 
Now we get the following:
\begin{align}
\begin{split}
\max_{\scenUnsteerable \in \BudgetedUncertainty_\indexUnsteerable} \sum_{k\in N(U)} \sum_{\demand \in N(k)} \scenUnsteerable_\demand \,\hat{\varUnsteerableDemand}_{\demand k}  
&= \underset{\scenUnsteerable \in \BudgetedUncertainty_\indexUnsteerable}{\max} \ \sum_{\demand \in \InducedNeighborhood} \scenUnsteerable_\demand \\
&= \min\left\lbrace \sum_{\demand \in \InducedNeighborhood} \upperUnsteerable_\demand,\; 
\demandBound_\indexUnsteerable- \sum_{\demand \in \Demands \setminus \InducedNeighborhood} \lowerUnsteerable_\demand\right\rbrace.
\end{split}
\label{linearReformulationUnSteerableDemand}
\end{align}

Substituting \eqref{linearReformulationSteerableDemand} and \eqref{linearReformulationUnSteerableDemand} into \eqref{SepOriginal}, we obtain the following reformulation of the separation problem: Is there a subset $U \subseteq \Demands$ such that
\begin{align}
\tag{\ref{SepOriginal}'}
\label{ReformulationSepOriginal}
\begin{split}
&\min \left\lbrace \sum_{\demand \in U} \upperSteerable_\demand,\; 
\demandBound_\indexSteerable- \sum_{\demand \in \Demands \setminus U} \lowerSteerable_\demand \right\rbrace + \min\left\lbrace \sum_{\demand \in \InducedNeighborhood} \upperUnsteerable_\demand,\; 
\demandBound_\indexUnsteerable- \sum_{\demand \in \Demands \setminus \InducedNeighborhood} \lowerUnsteerable_\demand\right\rbrace\\
&\hspace*{1.7cm} > \sum_{\location \in \ConsiderationSet_\Locations(U)} \capacitySession \,\hat{x}_\location   + \sum_{\practice \in \ConsiderationSet_\Practices(U)} \capacityPractice_\practice \;? 
\end{split}
\end{align}

We show, that deciding the separation problem is NP-complete by a reduction from subset sum inspired by the one in~\cite{KRUMKE2019845}.
\begin{figure}
	\centering
	\includegraphics{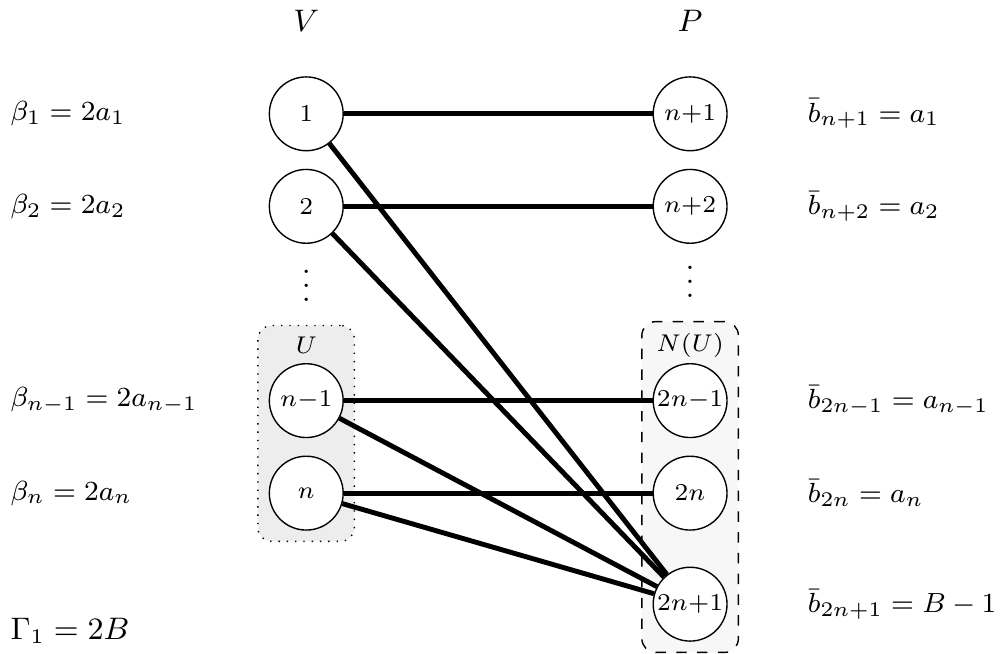}
	\caption{Constructed separation instance $\mathcal{I}'$ for given subset sum instance $\mathcal{I}=(A,B)$. Consideration sets are encoded by edges in bipartite graph.}
	\label{fig:reduction}
\end{figure}
\begin{theorem}
	The separation problem for $\GammaRobustBendersFormulation$ is NP-complete.
\end{theorem}
\begin{proof}
	To show the NP-completeness of the separation problem, we perform a reduction from the subset sum problem which is known to be NP-complete~\cite{10.5555/578533}.
	Let us recall the subset sum problem:
	Given a finite set $A=\{a_1,\dots, a_n\} \subseteq \N$ and an integer $B\in \N$, the subset sum problem asks whether there exists a subset $A'\subseteq A$ with $\sum_{a\in A'} a = B$.
	
	Given an instance $\mathcal{I}=(A, 
	B)$ of the subset sum problem, we construct an instance $\mathcal{I}'$ of the separation problem for $\GammaRobustBendersFormulation$ as follows:
	Let $\Demands = \{1,\dots, n\}$, $\Locations = \emptyset$, and $\Practices= \{n+1, \dots, 2n, 2n+1 \}$.  
	We set $\lowerSteerable_\demand = \lowerUnsteerable_\demand = \upperUnsteerable_\demand = 0$ for all $\demand
\in \Demands$, that is we do not consider unsteerable patient demands. 
    Moreover, we set $\upperSteerable_\demand= 2 a_\demand$ for all $\demand \in \Demands$.
    Concerning the practices' treatment capacities, we set $\capacityPractice_\practice= a_{\practice-n}$ for all $\practice \in \Practices \setminus \{2n+1\}$ and $\capacityPractice_{2n+1} = B-1$.
    The consideration sets are defined as $\ConsiderationSet(\demand)= \{\demand+n,\, 2n+1\}$ for all $\demand\in \Demands$ and we choose $\demandBound_\indexSteerable= 2B$. The construction of $\mathcal{I}'$ is visualized in Figure~\ref{fig:reduction}.
    
    For our choice of parameters, the separation problem for $\GammaRobustBendersFormulation$ reduces to: Is there a subset $U \subseteq \Demands$ such that $\min \left\lbrace \sum_{\demand \in U} \upperSteerable_\demand,\; 
    \demandBound_\indexSteerable \right\rbrace > \sum_{\practice \in \ConsiderationSet(U)} \capacityPractice_\practice$?
    
    We show that the constructed instance~$\mathcal{I}'$ of the separation problem is a yes-instance if and only if the subset sum instance~$\mathcal{I}$ is a yes-instance.
    
    First, assume that $\mathcal{I}$ is a yes-instance and let $A'\subseteq A$ with $\sum_{a\in A'} a = B$.
    Then for $U=\{\demand\in \Demands: a_\demand \in A'\}$ it holds that 
       \begin{align*}
    \min \left\lbrace \sum_{\demand \in U} \upperSteerable_\demand,\; 
    \demandBound_\indexSteerable \right\rbrace
    = \min \left\lbrace \sum_{a \in A'} 2a,\; 
    2B \right\rbrace = 2B > 2B-1 = \sum_{a\in A'} a + B-1 = \sum_{\practice \in \ConsiderationSet(U)} \capacityPractice_\practice  
    \end{align*}
    which shows that $\mathcal{I}'$ is a yes-instance.
    
    Conversely, assume that $\mathcal{I}'$ a yes-instance and let $U\subseteq \Demands$ be a subset of demand origins with $\min \left\lbrace \sum_{\demand \in U} \upperSteerable_\demand,\; 
    \demandBound_\indexSteerable \right\rbrace > \sum_{\practice \in \ConsiderationSet(U)} \capacityPractice_\practice$.
    We show that $A'=\{a_\demand \in A : \demand \in U\}$ satisfies $\sum_{a\in A'} a = B$.
    To that end, we begin by showing that 
    \begin{align}
    \sum_{\demand \in U} \upperSteerable_\demand \leq
    \demandBound_\indexSteerable \Leftrightarrow \sum_{a\in A'} 2a \leq 2B \Leftrightarrow \sum_{a\in A'} a \leq B.
    \label{ReductionP1}
    \end{align}
    Assume the contrary, i.e., that $\sum_{a\in A'} a > B$.
    Then by our choice of $U$, we have that
    \begin{align*}
    	\min \left\lbrace \sum_{\demand \in U} \upperSteerable_\demand,\; 
    	\demandBound_\indexSteerable \right\rbrace =  \demandBound_\indexSteerable > \sum_{\practice \in \ConsiderationSet(U)} \capacityPractice_\practice \Leftrightarrow 2B > \sum_{a\in A'} a + B -1 \Leftrightarrow \sum_{a\in A'} a \leq B
    \end{align*}
    which is a contradiction and thus proves \eqref{ReductionP1}.
    By our choice of $U$, we moreover get
      \begin{align}
    \min \left\lbrace \sum_{\demand \in U} \upperSteerable_\demand,\; 
    \demandBound_\indexSteerable \right\rbrace 
    =  \sum_{a\in A'} 2a > \sum_{\practice \in \ConsiderationSet(U)} \capacityPractice_\practice 
    &\Leftrightarrow \sum_{a\in A'} 2a > \sum_{a\in A'} a + B -1 \notag \\
    &\Leftrightarrow \sum_{a\in A'} a \geq B
    \label{ReductionP2}
    \end{align}
    Combining \eqref{ReductionP1} and \eqref{ReductionP2}, it follows that $\sum_{a\in A'} a = B$ and thus $\mathcal{I}$ is a yes-instance.
    
    Finally, we remark that the separation problem for $\GammaRobustBendersFormulation$ is contained in NP as we can compute all terms in \eqref{ReformulationSepOriginal} for given $U\subseteq \Demands$ in polynomial time.
\end{proof}

Just as in the deterministic setting, the separation problem for $\GammaRobustBendersFormulation$ is trivial if we only consider unsteerable demands due to Assumption~\ref{ass:1}.

To decide the separation problem, we propose an integer linear program based on formulation \eqref{ReformulationSepOriginal}.
This formulation requires variables to encode our choice of $U\subseteq \Demands$ as well as the derived sets $\ConsiderationSet(U) \subseteq \Locations \cup \Practices$ and  $\InducedNeighborhood \subseteq \Demands$.
Therefore, we introduce variables $\varEncodeU_\demand \in \{0,1\}$ that take the value one if demand origin $\demand \in \Demands$ is in the set $U$ and zero otherwise.
Variables $\varEncodeN_k \in \{0,1\}$ take the value one if treatment facility $k\in \Locations \cup \Practices$ is in the consideration set $N(U)$ and zero otherwise.
Finally, we introduce variables $\varEncodeV_\demand \in \{0,1\}$ that take the value one  if $v\in \InducedNeighborhood$ and zero otherwise. 
To linearize the inner minimization problems in \eqref{ReformulationSepOriginal}, we furthermore introduce continuous variables $d_\indexSteerable\geq 0$ and $d_\indexUnsteerable\geq 0$ which attain the value of the respective worst case patient demand for the chosen subset $U \subseteq \Demands$ in an optimal solution. We can now formulate the separation problem as follows:
\begin{maxi!}|s|[2]
	{d_\indexSteerable,d_\indexUnsteerable, \varEncodeU, \varEncodeN, \varEncodeV}
	{d_\indexSteerable + d_\indexUnsteerable - \sum_{\location \in \Locations} \capacitySession \hat{x}_\location \varEncodeN_\location - \sum_{\practice \in \Practices} \capacityPractice_\practice \varEncodeN_\practice \label{23a}}
	{}
	{\SeparationProblemFormulation\quad}
	\addConstraint{\varEncodeN_k}{\geq \varEncodeU_\demand}{\quad \forall \demand \in \Demands,\, k\in \ConsiderationSet(\demand) \label{23b}}
	\addConstraint{\varEncodeV_\demand}{\leq \sum_{\demand' \in \ConsiderationSet(\closestFixedFacility{\hat{\varUnsteerableDemand}}{\demand})} \varEncodeU_{\demand'}}{\quad \forall \demand \in \Demands \label{23c}}
	\addConstraint{d_\indexSteerable}{\leq \sum_{\demand \in \Demands} \upperSteerable_\demand \varEncodeU_\demand}{\label{23d}}
	\addConstraint{d_\indexSteerable}{\leq \demandBound_\indexSteerable - \sum_{\demand \in \Demands} \lowerSteerable_\demand (1- \varEncodeU_\demand)}{\label{23e}}
	\addConstraint{d_\indexUnsteerable}{\leq \sum_{\demand \in \Demands} \upperUnsteerable_\demand \varEncodeV_\demand}{\label{23f}}
	\addConstraint{d_\indexUnsteerable}{\leq \demandBound_\indexUnsteerable - \sum_{\demand \in \Demands} \lowerUnsteerable_\demand (1- \varEncodeV_\demand)}{\label{23g}}
	\addConstraint{\varEncodeU_\demand, \varEncodeV_\demand}{\in \{0,1\}}{\quad \forall \demand \in \Demands \label{23h}}
	\addConstraint{\varEncodeN_k}{\in \{0,1\}}{\quad \forall k \in \Locations \cup \Practices \label{23i}}
	\addConstraint{d_\indexSteerable,d_\indexUnsteerable}{\geq 0.}{\label{23j}}
\end{maxi!}
Thereby, inequalities \eqref{23b} enforce that $\varEncodeN_k$ for $k\in \Locations \cup \Practices$ encode the consideration set $\ConsiderationSet(U)$ and constraints \eqref{23c} ensure that $\varEncodeV_\demand$ for $\demand \in \Demands$ encode $\InducedNeighborhood$.
The remaining inequalities \eqref{23d}--\eqref{23g} model the reformulated inner minimization problems for the steerable and unsteerable patient demands derived in  \eqref{linearReformulationSteerableDemand} and \eqref{linearReformulationUnSteerableDemand}, respectively.

Given an optimal solution $(\hat{d}_\indexSteerable,\hat{d}_\indexUnsteerable, \hat{\varEncodeU}, \hat{\varEncodeN}, \hat{\varEncodeV})$ to $\SeparationProblemFormulation$, we can decide the separation problem as follows.
If the solution value of $(\hat{d}_\indexSteerable,\hat{d}_\indexUnsteerable, \hat{\varEncodeU}, \hat{\varEncodeN}, \hat{\varEncodeV})$ is non-positive, it follows that the optimal solution $(\hat{y}, \hat{x}, \hat{\varUnsteerableDemand}, \hat{\dualBudgetCap}, \hat{\kappa}, \hat{\rho})$ to the restricted master problem is also an optimal solution to $\GammaRobustBendersFormulation$.
Otherwise, we get the violating subset $\hat{U} \coloneqq \{\demand \in \Demands: \hat{\varEncodeU}_\demand=1\}$ which is added to $\SubsetOfSubsets$ and we iterate by resolving the restricted master problem.

%% file: multi_period.tex
\section{Disaggregation of sessions}
\label{sec:multi_period}
The previous sections considered the (robust) strategic planning problem for MMUs in a session-aggregated form.
That is, we modeled the patient demands at each demand origin through a single aggregated value and decided on the total number of MMU sessions at each MMU operation site.
Such an aggregation has several shortcomings, as it artificially smoothes out patient demands and entails a post-processing procedure in order to distribute the scheduled MMU operations throughout the week.

To overcome these drawbacks, we disaggregate the strategic planning problem for MMUs by considering session-specific demands, treatment capacities, and MMU operations.
Thereby, steerable patient demands are allowed to be assigned between sessions to balance out each session's workload.
The \textit{session-specific strategic planning problem for MMUs} then asks which MMU operation sites should be serviced in what sessions in order to satisfy all patient demands at minimum cost.

We formalize this problem by letting $\Sessions$ denote the sessions of the week which generally comprise a morning and an afternoon session for every working day of the week, i.e., $\Sessions=\{\text{MO}_{\text{AM}},\dots,\text{SAT}_{\text{PM}}\}$.
To model session-specific treatment capacities, we consider the session-expanded potential MMU operation sites $\eLocations\coloneqq \Locations \times \Sessions$ where each site $\elocation=(\location,\session) \in \eLocations$ can be serviced at most once.
Analogously, we consider the session-expanded practices $\ePractices = \Practices \times \Sessions$ with treatment capacity $\capacityPractice_{\epractice} \in \N$ for every $\epractice=(\practice,\session)\in \ePractices$.
This enables, e.g., the modeling of the common practice that German physicians are closed on Wednesday afternoons by setting $\capacityPractice_{(\practice,\text{WED}_{\text{PM}})}=0$ for all $\practice \in \Practices$.
The definition of session-expanded treatment facilities gives rise to the introduction of session-specific strategic MMU operation plans.

\begin{definition}
	A \textit{session-specific strategic MMU operation plan} is a function $\MMUoperation\colon \eLocations \to \{0,1\}$.
	The \textit{cost} of a session-specific strategic MMU operation plan $\MMUoperation$ is defined by the costs of setting up sites and operating MMU sessions, i.e., we have $c(\MMUoperation)\coloneqq \sum_{\location\in\Locations: \exists \session \in \Sessions : \MMUoperation_{(\location, \session)} > 0} \costLocation_\location +  \sum_{\elocation \in \eLocations} \costSession \, \MMUoperation_{\elocation}$.
\end{definition}

To model session-specific patient demands, we consider the session-expanded demand origins $\eDemands = \Demands \times \Sessions$ with a steerable patient demand $\steerableDemand_{\edemand} \in \N$ and an unsteerable patient demand $\unsteerableDemand_{\edemand} \in \N$ for each $\edemand\in \eDemands$.
While unsteerable patient demands immediately visit the closest considered operating treatment facility, steerable patient demands can be shifted between sessions.
Thus, we model two independent consideration sets for each $\edemand=(\demand,\session) \in \eDemands$:
a consideration set $\eSteerableConsiderationSet(\edemand) \subseteq \eLocations \cup \ePractices$ for the steerable patient demands, and a consideration set $\eUnsteerableConsiderationSet(\edemand) \subseteq (\Locations \cup \Practices) \times \{\session\}$ for the unsteerable patient demands.
As a result, we have to extend the definition of an assignment  of the steerable patient demands.

\begin{definition}
	A \textit{session-specific assignment} of the steerable patient demands is a set of functions $\left\lbrace \assignmentSteerablePatientDemand{\edemand}\right\rbrace _{\edemand \in\eDemands}$ with $\assignmentSteerablePatientDemand{\edemand} \colon \eSteerableConsiderationSet(\edemand) \to \N$ that distribute all steerable patient demands within their respective session-expanded consideration set, i.e., $\sum_{\boldsymbol{k} \in \eSteerableConsiderationSet(\edemand)} \assignmentSteerablePatientDemand{\edemand}(\boldsymbol{k})= \steerableDemand_{\edemand}$ for all $\edemand \in \eDemands$.
\end{definition}

Next, we define \textit{feasible} session-specific MMU operation plans.
To ease notation, let $\eSteerableConsiderationSet(\boldsymbol{k})\coloneqq 
\{\edemand\in \eDemands : \boldsymbol{k} \in \eSteerableConsiderationSet(\edemand) \}$ 
$\left( \eUnsteerableConsiderationSet(\boldsymbol{k})\coloneqq 
\{\edemand\in \eDemands : \boldsymbol{k} \in \eUnsteerableConsiderationSet(\edemand) \} \right) $ denote all session-expanded patient demand origins whose (un-)steerable  patient demands can target the treatment facility $\boldsymbol{k} \in 
\eLocations \cup \ePractices$.
Moreover, let $\eclosestFacility{m}{\edemand} \in \eUnsteerableConsiderationSet(\edemand)$ denote the closest considered operating treatment facility which is targeted by all unsteerable patient demands originating in $\edemand\in \eDemands$ for given session-specific MMU operation plan $\MMUoperation$.

\begin{definition}
	A session-specific strategic MMU operation plan $\MMUoperation$ is \textit{feasible} if there exists a session-specific assignment of the steerable patient demands $\left\lbrace \assignmentSteerablePatientDemand{\edemand}\right\rbrace _{\edemand \in\eDemands}$ that respects the session-specific treatment capacity at each treatment facility $\boldsymbol{k} \in \eLocations \cup \ePractices$, that is 
	\begin{align*}
	\sum_{\edemand \in \eDemands: \eclosestFacility{\MMUoperation}{\edemand}=\boldsymbol{k}} \unsteerableDemand_{\edemand} 
	+ \sum_{\edemand \in \eSteerableConsiderationSet(\boldsymbol{k})} \assignmentSteerablePatientDemand{\edemand}(\boldsymbol{k}) \leq \begin{cases}
	\capacityPractice_{\boldsymbol{k}} \quad &\text{if } \boldsymbol{k}\in \ePractices,\\
	\capacitySession \, \MMUoperation_{\boldsymbol{k}} & \text{if } \boldsymbol{k}\in \eLocations.
	\end{cases}
	\end{align*}
\end{definition}

Finally, we can employ the notion of a feasible session-specific MMU operation plan to formalize the definition of the session-specific strategic planning problem for MMUs.

\begin{definition}[$\sessionSpecificStrategicPlanningProblem$]
	Let $\Sessions$ denote the sessions of the week and let $\location\in \Locations$ be the potential MMU operation sites with setup costs $\costLocation_\location\in \N$.
	Moreover, let $\practice\in \Practices$ be the existing practices with treatment capacities $\capacityPractice_{(\practice,\session)}\in \N$ in session $\session \in \Sessions$.
	In every session $\session\in\Sessions$, each patient demand origin $\demand \in \Demands$ has steerable and unsteerable demands $\steerableDemand_{(\demand,\session)},\unsteerableDemand_{(\demand,\session)} \in \N$ that can be serviced within the consideration sets $\eSteerableConsiderationSet\left((\demand, \session)\right) \subseteq (\Locations\cup \Practices) \times \Sessions$ and $\eUnsteerableConsiderationSet\left((\demand, \session)\right) \subseteq (\Locations\cup \Practices) \times \{\session\}$, respectively.
	Then, the \textit{session-specific strategic planning problem for MMUs} ($\sessionSpecificStrategicPlanningProblem$) asks for a feasible session-specific strategic MMU operation plan of minimum cost, where every operated MMU session induces the cost $\costSession\in \N$ and yields a treatment capacity $\capacitySession\in \N$.
\end{definition}

As the $\sessionSpecificStrategicPlanningProblem$ only allows for a single MMU operation at every site $\elocation\in\eLocations$, the $\sessionSpecificStrategicPlanningProblem$ does not generalize the $\strategicPlanningProblem$ and thus the problem's strong NP-hardness does not follow from Theorem~\ref{NP_SPMMU}.
However, the reduction referenced in the proof of Theorem~\ref{NP_SPMMU} is still applicable for the $\sessionSpecificStrategicPlanningProblem$, as all subsets in this reduction are chosen at most once. 

\begin{theorem}
	The $\sessionSpecificStrategicPlanningProblem$ is strongly NP-hard.
\end{theorem}

Comparing the $\sessionSpecificStrategicPlanningProblem$ to the $\strategicPlanningProblem$, we can observe that both problems are closely related.
In the following, we devise an integer linear programming formulation for the $\sessionSpecificStrategicPlanningProblem$ which is nearly identical to formulation $\OriginalFormulation$ from Section~\ref{sec:prob_def} and emphasizes the problems' common structure.
Let variables $y_\location \in \{0,1\}$ indicate whether site $\location \in \Locations$ is set up, let variables $x_{\elocation} \in \{0,1\}$ decide whether site $\elocation \in \eLocations$ is serviced by an MMU, and let variables $z_{\edemand \boldsymbol{k}}\in \N$ determine the 
steerable demand originating in $\edemand \in \eDemands$ that is 
assigned to treatment facility $\boldsymbol{k} \in \eSteerableConsiderationSet(\edemand)$.
Moreover, let variables $\varUnsteerableDemand_{\edemand \boldsymbol{k}} \in \{0,1\}$ indicate the 
closest operating treatment facility $\boldsymbol{k} \in \eUnsteerableConsiderationSet(\edemand)$ that is 
targeted by all unsteerable demands originating in $\edemand \in \eDemands$. 
To that end, let $\eOrder{\edemand} \colon \{1,\dots, 
|\eUnsteerableConsiderationSet(\edemand)|\} \to \eUnsteerableConsiderationSet(\edemand)$ define an order on the  consideration set  $\eUnsteerableConsiderationSet(\edemand)$ that is non-decreasing with respect to the treatment facility's distance $\dist\colon \eDemands \times (\eLocations \cup \ePractices) \to \N$ to demand origin $\edemand\in \eDemands$.
As in Section~\ref{sec:prob_def}, we denote all MMU operation sites and practices within the consideration set of unsteerable demands at $\edemand \in \eDemands$ by
 $\eUnsteerableConsiderationSet_{\eLocations}(\edemand) \coloneqq \eUnsteerableConsiderationSet(\edemand)\cap \eLocations$ and
$\eUnsteerableConsiderationSet_{\ePractices}(\edemand) \coloneqq \eUnsteerableConsiderationSet(\edemand)\cap \ePractices$, respectively.
We can now formulate the $\sessionSpecificStrategicPlanningProblem$ as follows:

\begin{mini!}|s|[2]
	{y,x,z,\varUnsteerableDemand}
	{\sum_{\location\in \Locations} \costLocation_\location \, 
		y_\location + \sum_{\elocation \in \eLocations} \costSession \, 
		x_{\elocation} \label{}}
	{}{\OriginalSessionSpecificFormulation\;\;\,}
	\addConstraint{x_{\elocation}}{\leq y_\location }{\quad \forall \elocation=(\location,\session)\in \eLocations \label{}}
	\addConstraint{\sum_{\boldsymbol{k} \in \eSteerableConsiderationSet(\edemand)} z_{\edemand \boldsymbol{k}}}{\geq 
		\steerableDemand_{\edemand}}{\quad \forall \edemand\in \eDemands\label{ez1}}
	\addConstraint{\sum_{\edemand\in \eSteerableConsiderationSet(\elocation)} z_{\edemand \elocation}  + \sum_{\edemand\in \eUnsteerableConsiderationSet(\elocation)} \unsteerableDemand_{\edemand} \, \varUnsteerableDemand_{\edemand \elocation}}{\leq 
		\capacitySession \, x_{\elocation}}{\quad \forall \elocation\in \eLocations \label{ez2}}
	\addConstraint{\sum_{\edemand\in \eSteerableConsiderationSet(\epractice)} z_{\edemand \epractice}  + \sum_{\edemand\in \eUnsteerableConsiderationSet(\epractice)} \unsteerableDemand_{\edemand} \, \varUnsteerableDemand_{\edemand \epractice}}{\leq 
		\capacityPractice_{\epractice}}{\quad \forall \epractice\in \ePractices \label{ez3}}
	\addConstraint{\sum_{\boldsymbol{k}\in \eUnsteerableConsiderationSet(\edemand)} \varUnsteerableDemand_{\edemand \boldsymbol{k}}}{\geq 
		1}{\quad \forall \edemand\in \eDemands \label{}}
	\addConstraint{\varUnsteerableDemand_{\edemand \elocation}}{\leq x_{\elocation}}{\quad \forall \edemand 
		\in \eDemands,\, \forall \elocation \in \eUnsteerableConsiderationSet_{\eLocations}(\edemand)
		\label{}}
	\addConstraint{\varUnsteerableDemand_{\edemand \elocation}}{\geq x_{\elocation} - 
		\sum_{i=1}^{\einvorder{\edemand}{\elocation} -1 } \varUnsteerableDemand_{\edemand, 
			\eorder{\edemand}{i}}}{\quad \forall \edemand \in \eDemands,\, \forall \elocation \in  
		\eUnsteerableConsiderationSet_{\eLocations}(\edemand) \label{}}
	\addConstraint{\varUnsteerableDemand_{\edemand \epractice}}{\geq 1 - 
		\sum_{i=1}^{\einvorder{\edemand}{\epractice} -1} \varUnsteerableDemand_{\edemand, 
			\eorder{\edemand}{i}}}{\quad \forall \edemand \in \eDemands,\, \forall \epractice \in 
		\eUnsteerableConsiderationSet_{\ePractices}(\edemand)   \label{}}
	\addConstraint{x_{\elocation}\in \{0,1\},\; y_\location\in \{0,1\} }{}{\quad \forall \elocation=(\location, \session) \in 
		\eLocations}
	\addConstraint{\varUnsteerableDemand_{\edemand \boldsymbol{k}}\in \{0,1\}}{}{\quad \forall \edemand \in \eDemands,\, \forall \boldsymbol{k}\in 
		\eUnsteerableConsiderationSet(\edemand)}
	\addConstraint{z_{\edemand \boldsymbol{k}}\in \N}{}{\quad \forall \edemand \in \eDemands,\, \forall \boldsymbol{k}\in 
		\eSteerableConsiderationSet(\edemand). \label{ez4}}
\end{mini!}

Putting formulations $\OriginalFormulation$ and $\OriginalSessionSpecificFormulation$ side-by-side, we can confirm that the disaggreation of sessions leads to a structurally identical problem.
Consequently, all results from Section~\ref{sec:prob_def} can be directly transferred to $\sessionSpecificStrategicPlanningProblem$.
In particular, we can analogously show the correctness of the formulation $\OriginalSessionSpecificFormulation$.

\begin{theorem}
	$\OriginalSessionSpecificFormulation$ is an integer linear formulation for the $\sessionSpecificStrategicPlanningProblem$.
\end{theorem}

Moreover, we can apply the Benders decomposition approach from Section~\ref{sec:prob_def} to $\OriginalSessionSpecificFormulation$ to obtain the following analogous result.
\begin{theorem}
	Constraints \eqref{ez1}--\eqref{ez3}, \eqref{ez4} in $\OriginalSessionSpecificFormulation$ can be equivalently substituted by 
	\begin{align}
	&  \sum_{\edemand \in \boldsymbol{U}} \steerableDemand_{\edemand} + \sum_{\boldsymbol{k}\in \eSteerableConsiderationSet(\boldsymbol{U})} \sum_{\edemand \in \eUnsteerableConsiderationSet(\boldsymbol{k})} \unsteerableDemand_{\edemand} \,\varUnsteerableDemand_{\edemand \boldsymbol{k}} 
	\leq \sum_{\elocation \in \eSteerableConsiderationSet_{\eLocations}(\boldsymbol{U})} \capacitySession \,x_{\elocation}   + \sum_{\epractice \in \eSteerableConsiderationSet_{\ePractices}(\boldsymbol{U})} \capacityPractice_{\epractice} \quad  &\forall \boldsymbol{U}\subseteq \eDemands,
	\end{align}
	where	$\eSteerableConsiderationSet(\boldsymbol{U})\coloneqq  \bigcup_{\edemand \in \boldsymbol{U}} \eSteerableConsiderationSet(\edemand)$, $\eSteerableConsiderationSet_{\eLocations}(\boldsymbol{U})\coloneqq  \eSteerableConsiderationSet(\boldsymbol{U}) \cap \eLocations$, and $\eSteerableConsiderationSet_{\ePractices}(\boldsymbol{U})\coloneqq  \eSteerableConsiderationSet(\boldsymbol{U}) \cap \ePractices$ for $\boldsymbol{U}\subseteq \eDemands$.
\end{theorem}
\noindent
The resulting Benders reformulation of $\OriginalSessionSpecificFormulation$ will be denoted by $\SessionSpecificBendersFormulation$.

Finally, we note that we can obviously also transfer all results from Section~\ref{sec:uncertainty} to the $\sessionSpecificStrategicPlanningProblem$ to obtain a constraint generation procedure for the robust session-specific strategic planning problem for MMUs with interval and budgeted uncertainty sets.

%% file: computational_study.tex
\section{Computational study}
\label{sec:computational_study}
The computational study focuses on the $\strategicPlanningProblem$ as well as the $\robustStrategicPlanningProblem$ with budgeted  and interval uncertainty sets introduced in Sections~\ref{sec:prob_def} and \ref{sec:uncertainty}, respectively.
Based on a set of realistic test instances, we compare the cost and quality of the strategic MMU operation plans resulting from the three different approaches and investigate the so-called price of robustness.
To that end, Section~\ref{subsec:instances} elaborates on the design of our set of test instances before we describe the study design in Section~\ref{subsec:design}. The actual computational results of the study are presented in Section~\ref{subsec:study}.

\subsection{Test instances}
\label{subsec:instances}
The primary care system that provides the template for our test instances comprises three predominantly rural municipalities in western Germany. 
In the following, we successively consider the modeling of the practices $\Practices$, potential MMU operation sites $\Locations$, and patient demand origins $\Demands$.

 \paragraph{Practices}
According to data provided by the local department of public health for the year $2017$, there are $\num{20}$ primary care physicians with health insurance accreditation in the considered primary care system.
All physicians in the system operate in clinical sessions according to a weekly recurring schedule.
The official consultation hours of each clinical session are publicly available from the Association of Statutory Health Insurance Physicians Nordrhein~\cite{KVN}.
In addition to the official consultation hours, we assume that the first hour after the end of each clinical session serves as a buffer during which physicians no longer accept new patients, but continue treating existing ones.
To estimate each physician's weekly treatment capacity, we divide the total weekly consultation time (including buffers) by the average primary care physician consultation time of \SI{7.6}{\min} as reported for Germany in~\cite{Irvinge017902}.
After aggregating physicians that work in joint practices and ceiling the derived treatment capacities, this yields our set of $|\Practices|=16$ practices with treatment capacities $\capacityPractice_\practice\in [206,602]$ for all $\practice\in\Practices$; see Figure~\ref{fig:instance}.

 \paragraph{MMU operation sites}
Concerning the potential MMU operation sites $\Locations$, we evenly distribute $|\Locations|=28$ sites among the agglomerations of the considered municipalities; compare Figure~\ref{fig:instance}.
Under the assumption that MMUs operate Monday to Friday in a morning and afternoon session, we set $\capacityLocation_\location=10$ for all sites $\location\in \Locations$.
The duration of an MMU session is assumed to be $\SI{3.5}{\hour}$ which is slightly above the average session duration (without buffers) of $\SI{3.36}{\hour}$ observed for the physicians in the considered primary care system.
As MMUs may need to change their location between sessions, we do not anticipate buffers after MMU sessions.
By dividing the duration of an MMU session by the average German primary care physician's consultation time of \SI{7.6}{\min}~\cite{Irvinge017902}, we end up with a treatment capacity of $\capacitySession=28$ per operated MMU session.
With regard to the operation cost of MMUs, we assume that all sites are equally expensive to set up and that opening a new site is twice as undesirable as operating a weekly MMU session.
Thus, we choose the setup cost $\costLocation_\location = 2$ for all sites $\location \in \Locations$ and set the cost per operated MMU session to $\costSession=1$.

\begin{figure}[tb]
	\centering
	\includegraphics{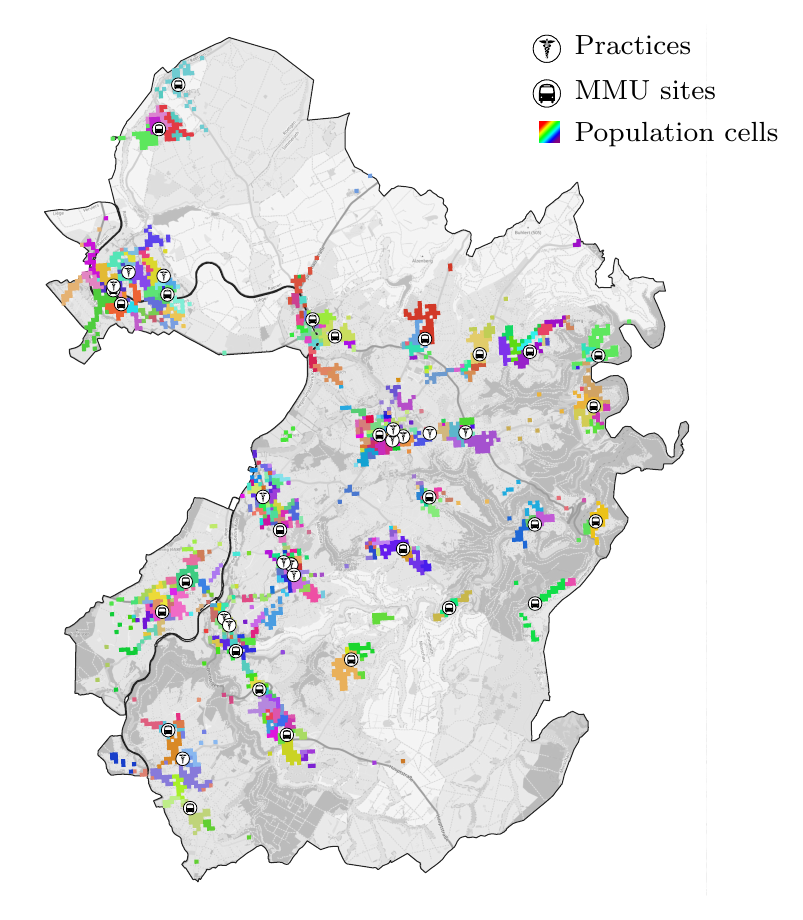}
	\caption[theCaption]{Locations of practices, potential MMU locations, and non-aggregated population cells clustered according to consideration sets for $\maxDist=\SI{6}{\kilo\meter}$. \footnotemark{} }
	\label{fig:instance}
\end{figure}

\footnotetext{Map tiles by Humanitarian OSM Team under CC0. Data by OpenStreetMap, under ODbL.}

\paragraph{Patient demand origins}
To model the patient demand origins $\Demands$, we rely on the population data determined by the latest German census conducted in $2011$~\cite{zensus2011}.
The census reports a total population of \num{35542} for the considered primary care system, specified at a resolution of $\num{2754}$ population cells measuring one hectare each.
To determine the consideration set of each population cell, we construct a street graph for the considered region based on map data from OpenStreetMap~\cite{osm} using OSMnx~\cite{BOEING2017126}.
The centers of the population cells and the treatment facilities are mapped to their respective closest node in the street network (as the crow flies), and we compute the driving distances between all population cells and treatment facilities along the street network.
The consideration set of each population cell is then defined as all treatment facilities that can be reached within a maximum driving distance $\maxDist\in\N$ in kilometers.
Thereby, we order consideration sets according to the distance between a treatment facility and the center of the population cell.
As a last step, we aggregate all population cells with identical (including order) consideration sets to obtain the set of demand origins~$\Demands$.
We note, that this aggregation and thus also the resulting set of patient demand origins $\Demands$ depends on the choice of the parameter $\maxDist\in\N$.
The non-aggregated population cells clustered according to their consideration sets for $\maxDist=\SI{6}{\kilo\meter}$ are exemplary shown in Figure~\ref{fig:instance}.

Next, we consider the steerable and unsteerable patient demands at each demand origin.
As empirical data concerning the primary care demand at each demand origin is unavailable, we have to rely on simulation to obtain rough estimates.
Specifically, we use the existing model of the considered primary care system in the hybrid agent-based simulation tool SiM-Care~\cite{comis2019patients} to obtain the number of primary care visits $\simulatedDemand{i}{\demand}\in \N$ per demand origin $\demand \in \Demands$ for every week $i\in\{1,\dots 52\}$ in a one year time horizon.
Both the simulation tool itself as well as the modeling of the considered primary care system are complex in nature and we refer to the aforementioned reference for further details.
Unsteerable patient demands, as considered here, are not representable in SiM-Care and we therefore assume that a fixed percentage $\fracUnsteerable \in [0,1]$ of the simulated primary care visits can be attributed to unsteerable patient demands.

In the deterministic setting of the $\strategicPlanningProblem$, we then choose the patient demands at each demand origin $\demand \in \Demands$ proportionately to the rounded average simulated demand $\avgSimulatedDemand{\demand}\coloneqq \round{\frac{1}{52} \sum_{i=1}^{52} \simulatedDemand{i}{\demand}}$ where $\text{round}\colon \R \to \Z$ denotes the rounding function $\round{x}\coloneqq \lfloor x + 0.5 \rfloor$. 
Specifically, we set
\begin{align*}
\unsteerableDemand_\demand = \round{\fracUnsteerable \, \avgSimulatedDemand{\demand}} \quad \text{and} \quad 
\steerableDemand_\demand = \avgSimulatedDemand{\demand} - \unsteerableDemand_\demand.
\end{align*}

In the uncertain setting of the $\robustStrategicPlanningProblem$, we choose the lower and upper bounds for the patient demands at each demand origin proportionately to the minimum and maximum simulated demands.
Formally, this translates into setting
\begin{align*}
\lowerUnsteerable_\demand = \roundFixBracket{ \fracUnsteerable\min_{1 \leq i \leq 52}\{ \simulatedDemand{i}{\demand}\}}
\quad \text{and} \quad
\lowerSteerable_\demand = \min_{1 \leq i \leq 52}\{ \simulatedDemand{i}{\demand}\} - \lowerUnsteerable_\demand
\end{align*}
for the lower bound of the unsteerable and steerable demands at $\demand\in\Demands$, and
\begin{align*}
\upperUnsteerable_\demand = \roundFixBracket{ \fracUnsteerable\max_{1 \leq i \leq 52}\{ \simulatedDemand{i}{\demand}\}} 
\quad \text{and} \quad
\upperSteerable_\demand = \max_{1 \leq i \leq 52}\{ \simulatedDemand{i}{\demand}\} - \upperUnsteerable_\demand
\end{align*}
for the respective upper bounds.

The budget parameters for the budgeted uncertainty sets are determined by the maximum simulated total demand among all weeks, i.e., we choose 
\begin{align*}
\demandBound_\indexUnsteerable = \roundFixBracket{\fracUnsteerable \max_{1\leq i \leq 52} \sum_{\demand \in \Demands} \simulatedDemand{i}{\demand}}
\quad \text{and} \quad 
\demandBound_\indexSteerable =  \bigg(\max_{1\leq i \leq 52} \sum_{\demand \in \Demands} \simulatedDemand{i}{\demand}\bigg) -\demandBound_\indexUnsteerable.
\end{align*}

For our test instances, we consider the patient demand origins obtained by choosing a maximum driving distance $\maxDist\in\{\num{6}, \num{7}, \dots, \num{11}\}$ and by varying the percentage of the unsteerable patient demands $\fracUnsteerable\in\{\num{0.2}, \num{0.25}, \dots, \num{0.45}\}$.
Table~\ref{table:instances} in Appendix~\ref{appendix:a} summarizes the characteristics of the resulting \num{36} test instances.
We remark, that the total average demand $\sum_{\demand \in \Demands} \steerableDemand_\demand {+} \unsteerableDemand_\demand$ as well as the total worst case demand $\sum_{\demand \in \Demands} \upperSteerable_\demand {+} \upperUnsteerable_\demand$  in each instance depends on the choice of $\maxDist \in \N$ as a result of the aggregation of population cells.
For our set of test instances, the total average demand is relatively robust towards changes in $\maxDist$, while the total worst case demand increases with $\maxDist$ due to the resulting larger number of demand origins $|\Demands|$.

\subsection{Implementation and computational setup}
\label{subsec:design}
In our computational study, we implemented all mathematical programs in Java using OpenJDK~\num{11}~\cite{openJDK} and the CPLEX~$12.8$ Java API~\cite{cplex}.
The CPLEX optimizer is restricted to one thread and all other CPLEX parameters are left at their default settings.
Instances of the $\strategicPlanningProblem$ are solved using formulation  $\BendersFormulation$ and instances of the  $\robustStrategicPlanningProblem$ with interval uncertainty sets are solved using formulation $\IntervalRobustBendersFormulation$.
To solve instances of the $\robustStrategicPlanningProblem$ with budgeted uncertainty sets,  we use formulation $\GammaRobustBendersFormulation$.
The separation problems in $\BendersFormulation$ and $\IntervalRobustBendersFormulation$ are solved using the LP-formulation from Appendix~\ref{appendix:separationLP}, and we integrate the separation procedure directly into the branch-and-bound scheme using lazy constraint callbacks.
The separation problem in $\GammaRobustBendersFormulation$ is solved using formulation $\SeparationProblemFormulation$ and we call the separation procedure only after the restricted master problem is solved to optimality.
Note, that the separation for $\GammaRobustBendersFormulation$ cannot be integrated into the branch-and-bound scheme as lazy constraint callbacks do not allow for the introduction of new variables.
As we cannot ensure that Assumption~\ref{ass:1} holds for our test instances, we apply the explicit enforcement from Appendix~\ref{appendix:enforcing_assumption}.
All computational experiments were performed on a cluster of machines running Ubuntu \num{18.04} with an Intel(R) Core(TM) i9-9900 CPU @ \SI{3.10}{\giga\hertz} and \SI{32}{\giga\byte} DDR$4$-Non-ECC main memory.
We restrict each individual job to one physical core and \SI{3.5}{\giga\byte} main memory.
All instances were solved to optimality (CPLEX default MIP gap tolerance $10^{-4}$) and all running times are reported in CPU seconds.

\subsection{Computational results}
\label{subsec:study}
The optimal objective values and CPU times of $\BendersFormulation$, $\GammaRobustBendersFormulation$, and $\IntervalRobustBendersFormulation$ for the \num{36} test instances are summarized in Table~\ref{table:comp} in Appendix~\ref{appendix:b}.
In the following, we discuss these results with a focus on the impact of the percentage of unsteerable patient demands $\fracUnsteerable\in[0,1]$ and the patients' maximum driving distance $\maxDist\in\N$. 
Furthermore, we investigate the \emph{price of robustness} --  a term introduced by Bertsimas and Sim~\cite{BertsimasPriceOfRobustness2004} that describes the additional cost of a robust solution compared to a non-robust solution that has to be payed for the protection against data uncertainties.
To that end, we compare the objective values of the robust solutions for $\GammaRobustBendersFormulation$ and  $\IntervalRobustBendersFormulation$ to the objective values of the nominal solutions for $\BendersFormulation$.

\input{figures/GraphsW.tex}

Examining the impact of the percentage of unsteerable patient demands, we visualize the objective function values of $\BendersFormulation$, $\GammaRobustBendersFormulation$, and $\IntervalRobustBendersFormulation$ for exemplary fixed $\maxDist\in \N$ and varying $\fracUnsteerable\in[0,1]$ in Figure~\ref{fig:ObjW}.
From the way we modeled the unsteerable patient demands, it is our expectation that a higher percentage of unsteerable patient demands leads to a higher objective value as a result of the associated loss of control over the patient demands.
Looking at Figure~\ref{fig:ObjW}, we can confirm this expectation regardless of the fixed maximum driving distance $\maxDist\in\N$ and the considered setting.
However, we can observe large differences in the degree of this effect.
In the robust setting with interval uncertainty sets $\IntervalRobustBendersFormulation$, the objective values are mostly unaffected by the choice of $\fracUnsteerable$.
This can be attributed to the conservatism of this approach, which leads to an overloading of the existing primary care system where the sheer level of demand seems to dominate the cost of the MMU operation plan.
In the deterministic setting $\BendersFormulation$, the influence of the percentage of the unsteerable patient demands is more pronounced, yet still relatively weak.
One explanation for this behavior is that the local level of unsteerable patient demands in this setting remains at a degree which can be mostly compensated by an appropriate reassignment of the steerable patient demands.
The greatest impact of the percentage of unsteerable patient demands $\fracUnsteerable\in [0,1]$ on the cost of the MMU operation plan can be observed in the robust setting with budgeted uncertainty sets for $\GammaRobustBendersFormulation$.
This showcases, that the budgeted uncertainty sets succeed at limiting the total patient demand as opposed to interval uncertainty sets, while still accounting for local worst-cases as opposed to the deterministic setting.
Concerning the price of robustness, we can confirm that the budgeted uncertainty sets manage to substantially lower the price of robustness compared to the interval uncertainty sets.
Furthermore, we can observe that the price of robustness is lowest for small $\fracUnsteerable\in [0,1]$ and increases with the percentage of the unsteerable patient demands.

To analyze the impact of the patients' maximum driving distance, we visualize the objective function values of $\BendersFormulation$, $\GammaRobustBendersFormulation$, and $\IntervalRobustBendersFormulation$ for exemplary fixed $\fracUnsteerable\in[0,1]$ and varying $\maxDist\in \N$ in Figure~\ref{fig:ObjD}. 
\input{figures/GraphsD.tex} %
Intuitively, one would assume that a higher patients' maximum driving distance leads to a lower objective function value as a result of larger consideration sets which yield more flexibility in the assignment of steerable patient demands.
However, looking at Figure~\ref{fig:ObjD} this assumption can only be verified for $\BendersFormulation$ and $\GammaRobustBendersFormulation$. 
For the choice of interval uncertainty sets in $\IntervalRobustBendersFormulation$, we can actually observe the opposite behavior.
Although this might seem counterintuitive at first glance, we can explain this behavior by the aggregation of the population cells during the instance generation process as described in Section~\ref{subsec:instances}:
An increase in $\maxDist$ leads to more diverse consideration sets which, in turn, result in a higher number of demand origins $|\Demands|$ as well as a higher total worst case patient demand $\sum_{\demand \in \Demands} \upperSteerable_\demand {+} \upperUnsteerable_\demand$; compare Table~\ref{table:instances} in Appendix~\ref{appendix:a}.
Paired with the previous observation that the objective value of $\IntervalRobustBendersFormulation$ seems to be dominated by this worst case demand, an increase in the objective value is actually to be expected.
In $\GammaRobustBendersFormulation$, we limit the total demand in the system through the use of budgeted uncertainty sets and thus can observe that the cost of MMU operation plans are decreasing in the patients' maximum driving distance.
The extent of the savings associated with an increase in $\maxDist$, decreases as we increase the percentage of the unsteerable patient demands $\fracUnsteerable$ for both $\BendersFormulation$ and $\GammaRobustBendersFormulation$.
This makes perfect sense as an increase in $\fracUnsteerable$ necessarily results in a smaller percentage of steerable patient demands for which we can actually profit from the enlarged consideration sets.
Concerning the price of robustness, also this collation of our results validates that the use of budgeted uncertainty sets reduces the price of robustness compared to the use of interval uncertainty sets.
Moreover, the difference in the optimal solution values between $\GammaRobustBendersFormulation$ and $\IntervalRobustBendersFormulation$ increases with the patients' maximum driving distance $\maxDist\in \N$, which is partly due to the undesired increase in the total worst case demand resulting from the aggregation of population cells during instance generation.

To justify why the price of robustness should be payed, we analyze the quality of the computed MMU operation plans.
For this purpose, we reuse the SiM-Care model of the considered primary care system which we referred to in Section~\ref{subsec:instances} to generate another set of patient demands for every week in a \num{10} year time horizon.
Thereby, we do not aggregate the population cells specified by the German census such that we end up with \num{520} weekly demands for each of the \num{2754} cells.
To determine the unsteerable patient demands for each realization, we flip a biased coin where we set the success probability to the percentage of unsteerable patient demands $\fracUnsteerable\in [0,1]$.
We analyze for each MMU operation plan and each realization the minimum total number of violations of the treatment capacities.
To clarify this metric, consider the following example: If \num{220} patients are assigned to a practice $\practice \in \Practices$ with weekly treatment capacity $\capacityPractice_\practice=200$,  this yields \num{20} violations.

Figure~\ref{fig:eval} shows the minimum total number of violations obtained by $\BendersFormulation$, $\GammaRobustBendersFormulation$, and $\IntervalRobustBendersFormulation$ for each of the \num{520} realizations and exemplary parameter choices.
\input{figures/Eval_Graphs.tex}
The first thing we want to emphasize, is that the solutions obtained using the robust models $\GammaRobustBendersFormulation$ and $\IntervalRobustBendersFormulation$ are actually feasible for all \num{520} realizations as there are no violations.
This feasibility of the robust solutions is not unique to the depicted parameter choices, but actually holds for all robust solutions we computed.
Considering the MMU operation plans obtained  using $\BendersFormulation$ and deterministic average demands, we can observe violations for quite a few realizations.
The number and extent of these violations depends on the percentage of unsteerable patient demands $\fracUnsteerable\in [0,1]$ and the patients' maximum driving distance $\maxDist \in \N$.
Specifically, the quality of the deterministic solutions deteriorates as we decrease $\maxDist$ and increase $\fracUnsteerable$ which seems reasonable as these are exactly the settings for which we observed the highest price of robustness.
Nevertheless, we must note that the highest number of violations observed over all parameter settings and realizations is \num{45}.
Setting \num{45} violations in relation to the total mean demand of roughly \num{3900}, only \SI{1}{\percent} of the demands cannot be accounted for by the deterministic solutions.

Although each violation can potentially lead to an avoidable emergency room visit, this is an admittedly good performance of the deterministic solutions.
A possible reason for this is the fact that SiM-Care does not feature an infectious model; compare~\cite{comis2019patients}.
Thus, the generated realizations do not show the local surges in demand resulting from the outbreak of an infectious disease.
To include these local demand spikes into our evaluation, we mimic infectious outbreaks in a very simplistic manner:
We select successively at random \num{5} of the \num{2754} population cells as outbreak centers and double the demands of each population cell within a \SI{1}{\km} radius of the outbreak center (as the crow flies).
Figure~\ref{fig:eval_infec} shows the minimum total number of violations obtained by $\BendersFormulation$, $\GammaRobustBendersFormulation$, and $\IntervalRobustBendersFormulation$ for each of the \num{520} realizations under the presence of infections outbreaks and exemplary parameter choices.
\input{figures/Eval_Graphs_infec.tex}
Looking at the results, we can observe that the local surges in demand lead to more violations in all solutions.
Evidently, the deterministic solutions perform worst, while the two robust approaches produce solutions of comparable quality.

While running times are clearly not a focus of this study, we note that all instances of $\BendersFormulation$ were solved within $\num{10}$ CPU seconds which is sufficiently quick for a planning problem at a strategic level and leaves room to consider even larger primary care systems.
The instances of $\GammaRobustBendersFormulation$, which are probably the most interesting ones from an application point of view, are noticeably more challenging than the instances of $\BendersFormulation$.
However, even the hardest instance solved within $\num{139}$ CPU seconds which is more than reasonable for this kind of strategic application.
Formulation $\IntervalRobustBendersFormulation$ is undoubtedly the most challenging among all considered formulations.
Especially for higher values of $\maxDist\in \N$, CPLEX struggles to close the MIP gap which results in a tailing-off phenomenon and running times of up to $\num{12925}$ CPU seconds.
Although this is considerably longer than the running times of the other formulations, even running times of this magnitude can still be deemed acceptable for a strategic planning problem. 
Moreover, we note that the running times for the $\robustStrategicPlanningProblem$ with interval uncertainty sets and high values of $\maxDist\in \N$ can be substantially improved by using a variation of formulation $\OriginalFormulation$ instead of $\IntervalRobustBendersFormulation$ that we decided to omit as we consider this setting for reference rather than as a serious alternative for real-world application.

%% file: figures/GraphsW.tex
\begin{figure}
	\centering
	\includegraphics{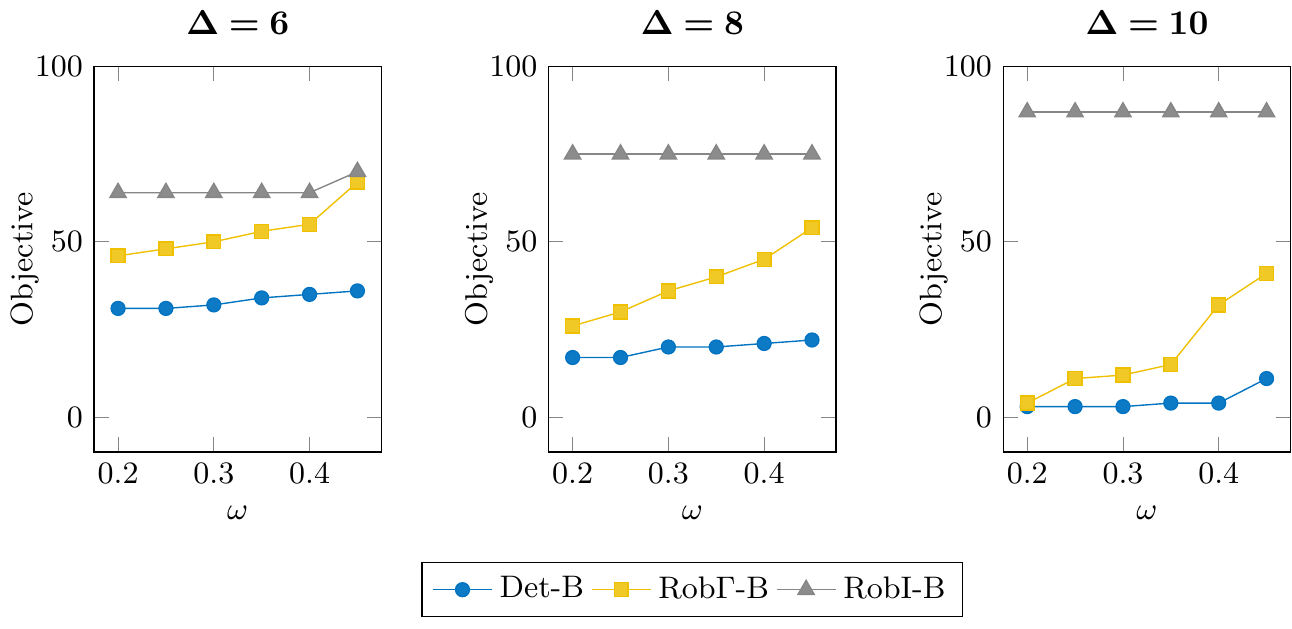}
	\caption{Objectives for fixed $\maxDist\in\{6,8,10\}$ and varying $\fracUnsteerable\in\{\num{0.2}, \dots, \num{0.45}\}$.}
	\label{fig:ObjW}
\end{figure}

%% file: figures/GraphsD.tex
\begin{figure}
\centering
\includegraphics{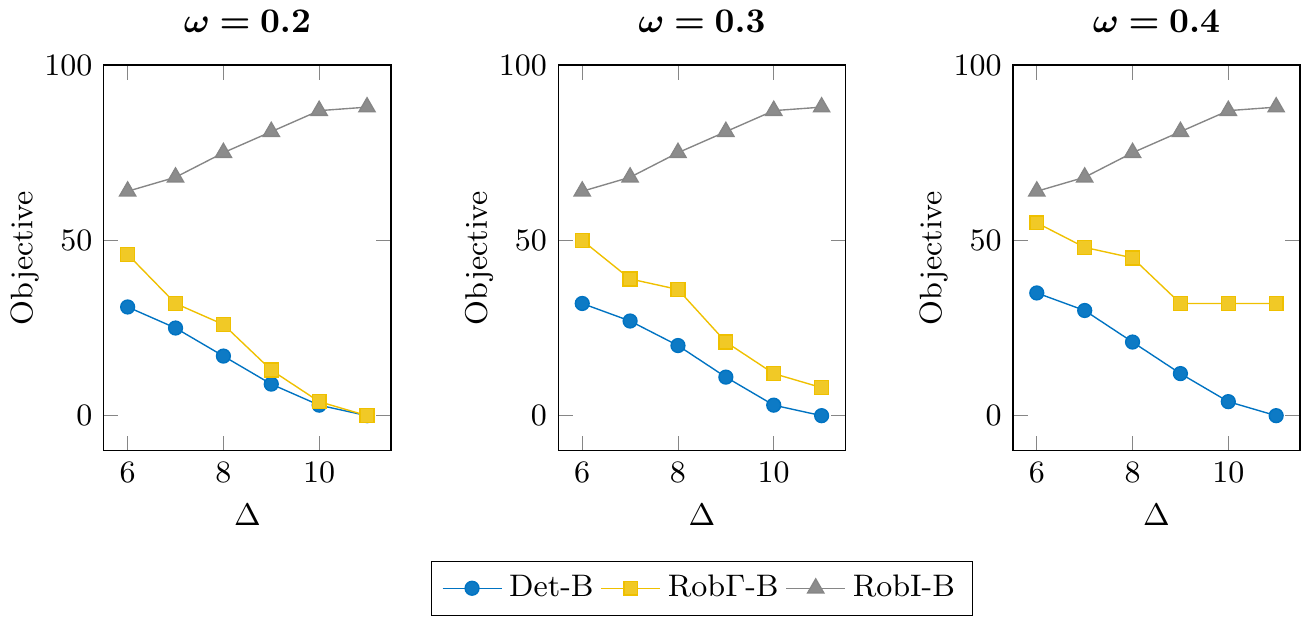}
\caption{Objectives for fixed $\fracUnsteerable\in\{\num{0.2}, \num{0.3}, \num{0.4}\}$ and varying $\maxDist\in\{6,\dots,11\}$.}
\label{fig:ObjD}
\end{figure}

%% file: figures/Eval_Graphs.tex
\begin{figure}
	\centering
	\includegraphics{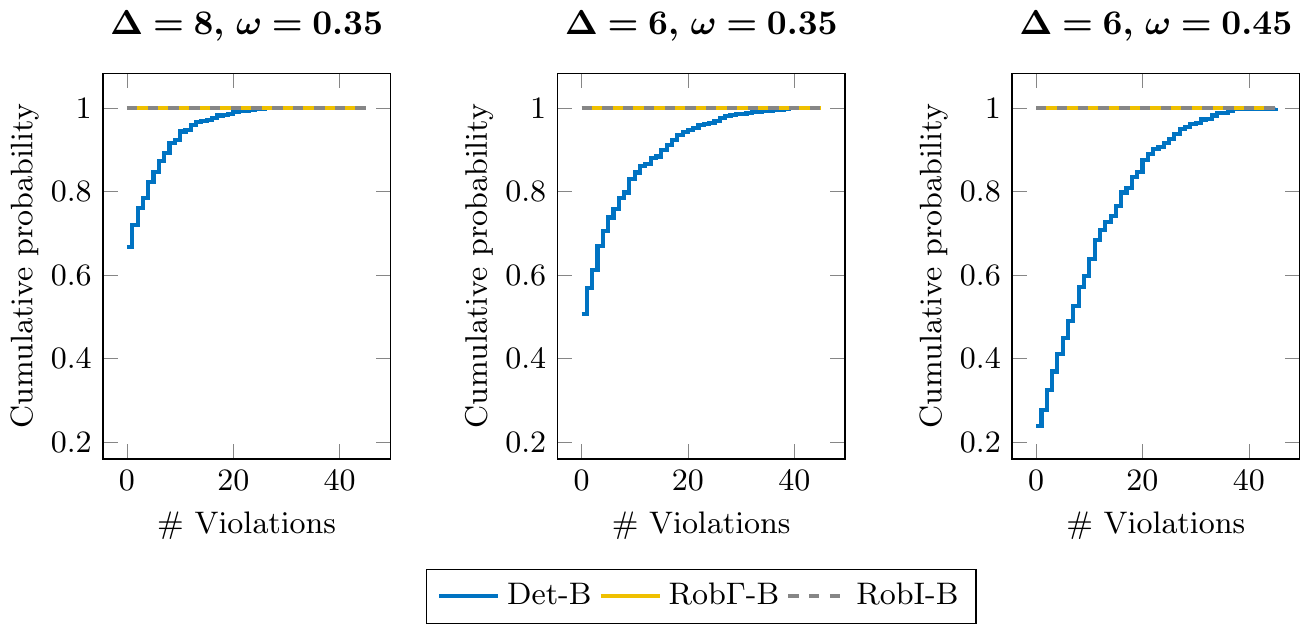}	
	\caption{Empirical distribution function of the minimum total number of violations for \num{520} realizations and parameter choices $(\maxDist, \fracUnsteerable) \in \{(8,0.35), (6,0.35), (6,0.45)\}$.}
	 
	\label{fig:eval}
\end{figure}

%% file: figures/Eval_Graphs_infec.tex
\begin{figure}
	\centering
	\includegraphics{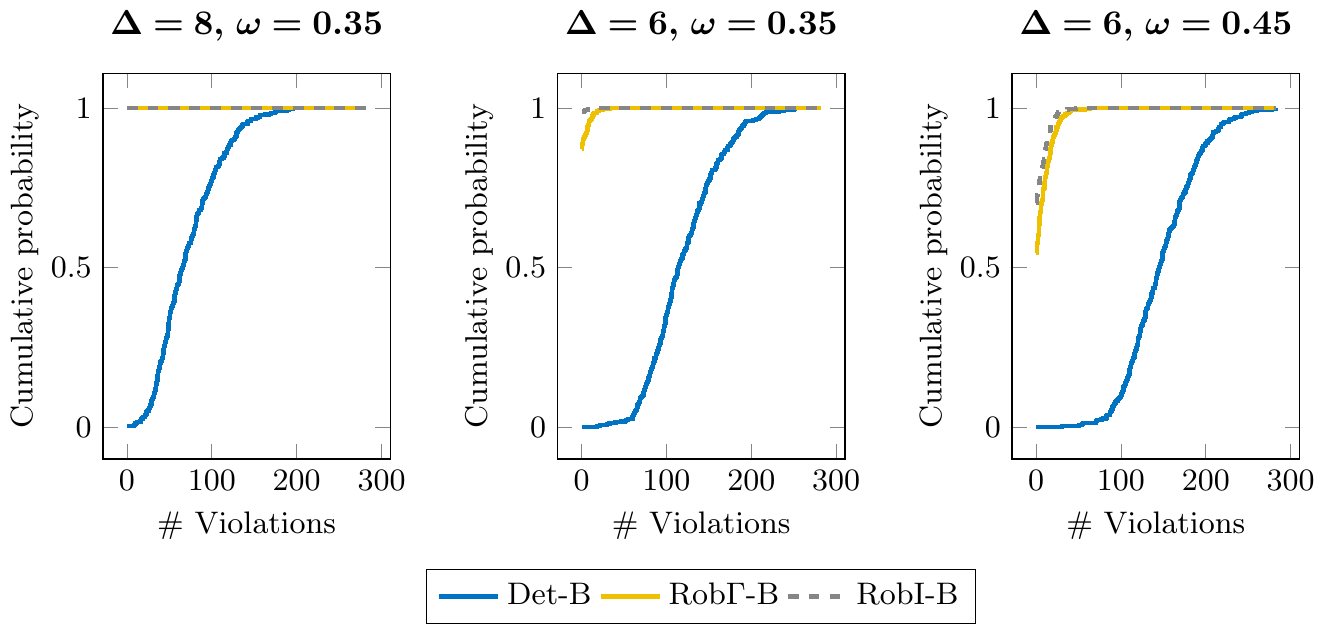}	
	\caption{Empirical distribution function of the minimum total number of violations for \num{520} realizations with \num{5} infectious outbreaks and $(\maxDist, \fracUnsteerable) \in \{(8,0.35), (6,0.35), (6,0.45)\}$.}
	\label{fig:eval_infec}
\end{figure}

%% file: conclusion.tex
\section{Discussion and conclusion}
\label{sec:conclusion}
In this paper, we studied the strategic planning problem for MMUs as a capacitated set covering problem.
As a new modeling concept, we considered existing infrastructure in the form of practices and both steerable and unsteerable patient demands.
While steerable patient demands can be assigned to any acceptable treatment facility, unsteerable demands will always visit the closest available treatment facility.
The interplay of the two types of demands brings a new aspect to location planning that has, to the best of our knowledge, not yet been considered in the literature.
We formulated the problem as a compact integer linear program and showed how this formulation can be solved by Benders decomposition and constraint generation.
Recognizing the importance of uncertainties in health care planning, we extended the problem to uncertain patient demands.
Using methods from robust optimization, we devised exact solution methods based on constraint generation for interval and budgeted uncertainty sets. 
%For the latter, the presence of unsteerable patient demands led to an NP-complete separation problem.
Since all these models consider patient demands and treatment capacities in a session-aggregated form that artificially smooths both out over the week, we subsequently introduced a session-specific version of the problem and showed that all our results transfer to it.
Finally, we conducted an extensive computational study in a real-world inspired setting.

Our study shows that the formulations presented in this paper enable us to compute optimized strategic MMU operation plans for real-world sized instances in an acceptable time frame.
The cost of the MMU operation and thus the cost of the provision of primary care significantly depends on three factors:
the percentage of unsteerable patient demands, the modeled consideration sets, and the handling of data uncertainties.
Thus, a key finding of this paper is the insight that these three factors can have a major impact on the resulting MMU operation plans and should thus be taken into account for strategic MMU operation planning.
From our computational experiments, we infer that $\GammaRobustBendersFormulation$ may be the most suitable formulation for this purpose. 
The formulation allows for the consideration of demand uncertainties while limiting the conservatism of solutions through the use of budgeted uncertainty sets. 
In addition, it is possible to trade off operational cost against robustness towards demand uncertainties by adjusting the budget parameters $\demandBound_\indexSteerable$ and $\demandBound_\indexUnsteerable$. %, which was not considered in our study. 

While this represents a major step forward for strategic MMU operation planning, we must not overlook the major limitations of our models and study that stem from our assumptions and open up directions for further research.
Concerning the limitations of our computational study, let us first note that the assessment of a physician's treatment capacity is a very delicate and personal matter that is beyond our field of expertise.
So while our estimates may not be completely unrealistic, we do not make any claims of correctness and would definitely recommend surveying each physician individually.  
Similar limitations apply to the patient demand origins and corresponding patient demands, where the lack of empirical data forced us to use a simulation model which can provide rough estimates at best.
Moreover, we have seen that the aggregation of patients to demand origins is highly non-trivial and may result in undesired behaviors such as an increase in the total worst-case patient demand.
With regard to model limitations, we want to note that our assumption that each demand origin's unsteerable patient demands target the same treatment facility is quite strict and definitely not true in reality.
To overcome this limitation, future work should investigate whether this assumption can be weakened, e.g., by assuming that the unsteerable patient demands target the three closest facilities in some fixed ratio.
Furthermore, we assume that the patient demands at the demand origins are independent which is questionable in practice to say the least.
While it could be difficult to remove this assumption entirely from our models, one step into this direction could start by considering steerable and unsteerable patient demands as being dependent.
In line with this goal, we could model one joint uncertainty set for the steerable and unsteerable patient demands instead of two separate ones to incorporate their dependencies into our models.

Summing up, we are confident that our models produce strategic MMU operation plans that can serve as a sound basis for an actual real-world implementation. 
That being said, we strongly recommend an expert validation of all plans prior to their implementation due to the discussed limitations.

\section{Acknowledgments}
This work was supported by the Freigeist-Fellowship of the Volkswagen Stiftung; the German research council (DFG) Research Training Group 2236 UnRAVeL; and the German Federal Ministry of Education and Research (grants no.~05M16UKC, 05M16PAA) within the project ``HealthFaCT - Health: Facility Location, Covering and Transport''.

%% file: appendix.tex
\begin{appendices}
	\section{Correctness of formulation \texorpdfstring{$\OriginalFormulation$}{(Det)}}	
	\label{appendix:detcorrect}
	We show that the integer linear program $\OriginalFormulation$ presented in Section~\ref{sec:prob_def} is indeed a formulation for the \mbox{$\strategicPlanningProblem$}, i.e., we formally prove Theorem~\ref{thm_first_formulation}.
	To that end, we prove that there always exists an optimal solution $(y,x,z,\varUnsteerableDemand)$ to $\OriginalFormulation$ in which all unsteerable demands originating in $\demand \in \Demands$ target the closest operated treatment facility $\closestFacility{y}{\demand}\coloneqq\arg\min_{k\in \ConsiderationSet(\demand): k \in \Practices \lor (k\in \Locations \land y_k=1)} \dist(\demand,k)$ in the consideration set.
	%where  $\AvlConsiderationSet{y}(v)\coloneqq \{k \in \ConsiderationSet(\demand) : k \in \Practices \lor (k\in \Locations \land y_k=1) \}$ denotes the operated treatment facilities within the consideration set of demand origin $\demand \in \Demands$. 
	
	\begin{lemma}
		\label{lem:feasible_1}
		Given a feasible solution $(y,x,z,\varUnsteerableDemand)$ to 
		$\OriginalFormulation$, 
		we can compute a feasible solution $(y,x,z,\varUnsteerableDemand^\prime)$ 
		to $\OriginalFormulation$ with the same objective value
		such that for every demand origin $\demand \in \Demands$ there is at most 
		one treatment facility $k\in \ConsiderationSet(\demand)$ with 
		$\varUnsteerableDemand^\prime_{\demand k} = 1$, i.e., $\sum_{k \in 
			\ConsiderationSet(\demand)} \varUnsteerableDemand'_{\demand k} \leq 1$, 
		in linear time.
	\end{lemma}
	\begin{proof}
		Given a feasible solution $(y,x,z,\varUnsteerableDemand)$ to 
		$\OriginalFormulation$, we set 
		
		\begin{align*}
		\varUnsteerableDemand^\prime_{\demand k} = \begin{cases}
		0 &\text{if}\quad \exists k'\in \ConsiderationSet(\demand) : \invorder{\demand}{k'} <  \invorder{\demand}{k} \land \varUnsteerableDemand_{\demand k^\prime}=1,\\
		\varUnsteerableDemand_{\demand k} & \text{else}.
		\end{cases}
		\end{align*}
		Clearly, $\varUnsteerableDemand'$ satisfies $\sum_{k \in 
			\ConsiderationSet(\demand)} \varUnsteerableDemand'_{\demand k} \leq 1$ and  can be computed in linear time.
		It holds that $\varUnsteerableDemand' \leq \varUnsteerableDemand$ and thus $(y,x,z,\varUnsteerableDemand')$ satisfies inequalities \eqref{4} and \eqref{5}.
		Moreover, the solution $(y,x,z,\varUnsteerableDemand')$ obviously satisfies constraints \eqref{6} and \eqref{7}. 
		Concerning constraints \eqref{8}, assume there exist $\demand \in \Demands$ and $k \in \ConsiderationSet_{\Locations}(\demand) $ such that 
		\begin{align}
		&\varUnsteerableDemand^\prime_{\demand k} < y_k - 
		\sum_{i=1}^{\invorder{\demand}{k} -1 } \varUnsteerableDemand'_{\demand,\order{\demand}{i}} \\
		\Leftrightarrow \;& y_k = 1 \land \sum_{i=1}^{\invorder{\demand}{k} } \varUnsteerableDemand'_{\demand,\order{\demand}{i}}=0.
		\label{assumption}
		\end{align}
		We can conclude from $y_k = 1$ and the feasibility of $(y,x,z,\varUnsteerableDemand)$ that $\sum_{i=1}^{\invorder{\demand}{k}} \varUnsteerableDemand_{\demand,\order{\demand}{i}}>0$, which necessitates that $\sum_{i=1}^{\invorder{\demand}{k} } \varUnsteerableDemand'_{\demand,\order{\demand}{i}}>0$.
		But this is a contradiction to assumption~\eqref{assumption} and completes the proof that $(y,x,z,\varUnsteerableDemand')$ satisfies constraints \eqref{8}.
		The validity of inequalities \eqref{9} can be shown analogously.
		Hence, $(y,x,z,\varUnsteerableDemand')$ is a feasible solution with the same 
		objective value as $(y,x,z,\varUnsteerableDemand)$.
	\end{proof}
	
	Based on this insight, we can now show that if $\OriginalFormulation$ is feasible, there always exists an optimal solution in which all unsteerable patient demands target their closest operated treatment facility.
	
	\begin{lemma}
		\label{lem:feasible_2}
		Let $(y,x,z,w)$ be a feasible solution to $\OriginalFormulation$ with $\sum_{k \in 
			\ConsiderationSet(\demand)} \varUnsteerableDemand_{\demand k} \leq 1$ for every 
		demand origin $\demand\in\Demands$. Then for all $\demand\in\Demands$ and $k\in \ConsiderationSet(\demand)$,
		$\varUnsteerableDemand_{\demand k}=1$ if and only if $k$ is $\demand$'s closest 
		operating treatment facility, i.e., $k=\closestFacility{y}{\demand}$.
	\end{lemma}
	\begin{proof}
		Let $\demand\in\Demands$ be a demand origin with closest operating treatment facility $\closestFacility{y}{\demand}=\order{\demand}{i} \in \ConsiderationSet(\demand)$, $i \in \{1, \dots, |\ConsiderationSet(\demand)|\}$.
		As all facilities $\order{\demand}{j}$ for $j < i$ are unoperated MMU sites by the definition of $i$, i.e., 
		$\order{\demand}{j} \in \ConsiderationSet_\Locations(\demand)$ with $y_{\order{\demand}{j}}=0$, we get that
		$\varUnsteerableDemand_{\demand \order{\demand}{j}}=0$ for all $j <i$ by \eqref{7}.
		The feasibility of $(y,x,z,w)$ now yields 
		\begin{align*}
		\varUnsteerableDemand_{\demand, \closestFacility{y}{\demand}} \geq 1 - \sum_{j=1}^{i -1 } \varUnsteerableDemand_{\demand,\order{\demand}{j}} = 1.
		\end{align*}
		Conversely, let $\varUnsteerableDemand_{\demand k} = 1$ for some $\demand\in\Demands$ and some $k\in \ConsiderationSet(\demand)$. %By constraint \eqref{7}, it follows that $y_k=1$ if $k\in \ConsiderationSet(\demand) \cap \Locations$.
		Assume $k$ is not the closest operating treatment facility to $v$, i.e., for the closest operated treatment facility $\closestFacility{y}{\demand}= \order{\demand}{i} \in \ConsiderationSet(\demand)$ holds $i < \invorder{\demand}{k}$.
		As we have $\sum_{j=1}^{\ConsiderationSet(\demand)} \varUnsteerableDemand_{\demand, \order{\demand}{j}} \leq 1$, it directly follows that $\varUnsteerableDemand_{\demand \order{\demand}{j}} = 0$ for all $j \neq  \invorder{\demand}{k}$.
		However, this implies that 
		\begin{align*}
		\varUnsteerableDemand_{\demand, \closestFacility{y}{\demand}} = 0 < 1 = 1 - \sum_{j=1}^{i -1 } \varUnsteerableDemand_{\demand,\order{\demand}{j}},
		\end{align*}
		which yields a violation of \eqref{8} or \eqref{9} for $v$ and $\closestFacility{y}{\demand}$, which in turn is a contradiction to the feasibility of $(y,x,z,w)$.
	\end{proof}
	
	Concerning the steerable patient demand, we have to show that if $\OriginalFormulation$ is feasible, there always exists an optimal solution in which no more than the steerable patient demand $\steerableDemand_\demand \in \N$ originates in each demand origin $\demand \in \Demands$.
	\begin{lemma}
		Given a feasible solution $(y,x,z,\varUnsteerableDemand)$ to 
		$\OriginalFormulation$, 
		we can compute a feasible solution $(y,x,z^\prime,\varUnsteerableDemand)$ 
		to $\OriginalFormulation$ with the same objective value and $\sum_{k\in 
			\ConsiderationSet(\demand)} z^\prime_{\demand k} = 
		\steerableDemand_\demand$ for all $\demand \in \Demands$ in linear time.
		\label{lem:steerable_demand}
	\end{lemma}
	\begin{proof}
		If for some $\demand \in \Demands$ it holds that $d_v^+ \coloneqq \sum_{k\in \ConsiderationSet(\demand)} z_{\demand k} - \steerableDemand_\demand > 0$, we can arbitrarily reduce the assigned steerable patient demands by $d_v^+$ without violating constraints \eqref{4} or \eqref{5}, e.g., by setting
		\begin{align*}
		z_{\demand k}^\prime \coloneqq z_{\demand k} - \min\left\lbrace z_{\demand k},\, \max \left\lbrace 0,\, \steerableDemand^+_\demand - \sum_{i=1}^{\invorder{\demand}{k}-1} z_{\demand \order{\demand}{i}}\right\rbrace \right\rbrace. 
		\end{align*}
	\end{proof}
	
	The correctness of $\OriginalFormulation$ for the $\strategicPlanningProblem$ is now immediate.
	
	\DetCor*
	\begin{proof}
		Given an optimal solution $(y,x,z,w)$ to $\OriginalFormulation$, a strategic MMU operation plan $\MMUoperation$ of minimum cost can be defined via $\MMUoperation_\location \coloneqq x_\location$ for all $\location \in \Locations$.
		By Lemma~\ref{lem:steerable_demand}, we can assume w.l.o.g.~that $\sum_{k\in \ConsiderationSet(\demand)} z_{\demand k} = \steerableDemand_\demand$ and thus $\{\assignmentSteerablePatientDemand{\demand}\}_{\demand \in \Demands}$ defined via $\assignmentSteerablePatientDemand{\demand}(k)\coloneqq z_{\demand k}$ for all $k\in \ConsiderationSet(\demand)$ induces an assignment of the steerable patient demands.
		The feasibility of $\MMUoperation$ now follows directly from Lemmata~\ref{lem:feasible_1} and~\ref{lem:feasible_2} for the assignment $\{\assignmentSteerablePatientDemand{\demand}\}_{\demand \in \Demands}$.	
	\end{proof}
	
	\section{Enforcement of Assumption~\ref{ass:1}}
	\label{appendix:enforcing_assumption}
	In this paper, we have exclusively worked under Assumption~\ref{ass:1}, i.e., that the residual treatment capacities $\helper_k$ for $k\in \Locations \cup \Practices$ are non-negative.
	However, as pointed out in Section~\ref{sec:prob_def}, this does not hold in general which would invalidate our results.
	Therefore, we have to explicitly enforce  Assumption~\ref{ass:1} in the master problem $\MasterProblem$ and thus also in the Benders formulation $\BendersFormulation$ through the following set of constraints:
	\begin{align} 
	&\sum_{\demand\in\ConsiderationSet(\location)}\unsteerableDemand_\demand\, \varUnsteerableDemand_{\demand 
		\location} \leq 	\capacitySession\, x_\location & \forall \location\in\Locations \label{detAss1}\\
	&\sum_{\demand\in\ConsiderationSet(\practice)}\unsteerableDemand_\demand\,
	\varUnsteerableDemand_{\demand 
		\practice} \leq \capacityPractice_\practice  & \forall \practice\in\Practices. \label{detAss2}
\intertext{	
	\indent As we start to model patient demands as random variables and consider robust formulations of the $\strategicPlanningProblem$, the definition of the residual treatment capacities for fixed first-stage decisions $\hat{\varUnsteerableDemand}$ and $\hat{x}$ has to be adjusted.
	That is, we define
}		
	\helper_\location&\coloneqq\capacitySession\, \hat{x}_\location - \max_{\scenUnsteerable \in \Uncertainty_\indexUnsteerable} 
	\sum_{\demand\in\ConsiderationSet(\location)}\scenUnsteerable_\demand\, \hat{\varUnsteerableDemand}_{\demand 
		\location} & \forall \location\in\Locations, \notag\\
	\helper_\practice&\coloneqq\capacityPractice_\practice 
	-\max_{\scenUnsteerable \in \Uncertainty_\indexUnsteerable}\sum_{\demand\in\ConsiderationSet(\practice)}\scenUnsteerable_\demand\,
	\hat{\varUnsteerableDemand}_{\demand 
		\practice} & \forall \practice\in\Practices. \notag
\intertext{
	To enforce Assumption~\ref{ass:1} in $\RobustBendersFormulation$, we thus have to consider the robust counterparts of inequalities \eqref{detAss1} and \eqref{detAss2} given by
}	 
		\max_{\scenUnsteerable \in \Uncertainty_\indexUnsteerable}&\sum_{\demand\in\ConsiderationSet(\location)}\scenUnsteerable_\demand\, \varUnsteerableDemand_{\demand 
			\location} \leq 	\capacitySession\, x_\location & \forall \location\in\Locations \label{robAss1}\\
		\max_{\scenUnsteerable \in \Uncertainty_\indexUnsteerable}&\sum_{\demand\in\ConsiderationSet(\practice)}\scenUnsteerable_\demand\,
		\varUnsteerableDemand_{\demand 
			\practice} \leq \capacityPractice_\practice  & \forall \practice\in\Practices.\label{robAss2}
		\end{align}
Inequalities $\eqref{robAss1}$ and $\eqref{robAss2}$ are non-linear in general, and have to be reformulated in a linear way for each specific choice of the consideration set $\Uncertainty_\indexUnsteerable$.

For interval scenarios, i.e., $\Uncertainty_\indexUnsteerable= \AllScenUnsteerable$, this is relatively straight forward as the unsteerable patient demands at each demand origin assume their upper bound which yields	
			\begin{align} 
			&\sum_{\demand\in\ConsiderationSet(\location)}\upperUnsteerable_\demand\, \varUnsteerableDemand_{\demand 
				\location} \leq 	\capacitySession\, x_\location & \forall \location\in\Locations\\
			&\sum_{\demand\in\ConsiderationSet(\practice)}\upperUnsteerable_\demand\,
			\varUnsteerableDemand_{\demand 
				\practice} \leq \capacityPractice_\practice  & \forall \practice\in\Practices.
			\end{align}

For budgeted uncertainty sets, i.e., $\Uncertainty_\indexUnsteerable= \BudgetedUncertainty_\indexUnsteerable$, things get slightly more complicated although we can essentially mimic our approach from Section~\ref{sec:uncertainty}.
That is, we formulate 
\begin{align*}
	\max_{\scenUnsteerable \in \BudgetedUncertainty_\indexUnsteerable}&\sum_{\demand\in\ConsiderationSet(k)}\scenUnsteerable_\demand\, \hat{\varUnsteerableDemand}_{\demand 
		\location}
\end{align*}
for fixed $k\in \Locations \cup \Practices$ and fixed $\hat{\varUnsteerableDemand}_{\demand k}\in \{0,1\}$ for all $\demand \in \Demands$ and $k \in \ConsiderationSet(\demand)$ via the following linear program:
			
			\begin{maxi*}|s|[2]
				{\scenUnsteerable}
				{\sum_{\demand \in N(k)} \scenUnsteerable_\demand \,\hat{\varUnsteerableDemand}_{\demand k}}
				{}
				{\LPPrimalAssumption{\hat{\varUnsteerableDemand}}\quad}
				\addConstraint{\scenUnsteerable_\demand}{\leq \upperUnsteerable_\demand}{\quad \forall \demand \in \Demands }
				\addConstraint{-\scenUnsteerable_\demand}{\leq -\lowerUnsteerable_\demand}{\quad \forall \demand \in \Demands }
				\addConstraint{\sum_{\demand \in \Demands} \scenUnsteerable_\demand}{\leq \demandBound_\indexUnsteerable}{ }
				\addConstraint{\scenUnsteerable_\demand}{\geq 0}{\quad\forall \demand \in \Demands.}
			\end{maxi*}
The dual problem of $\LPPrimalAssumption{\hat{\varUnsteerableDemand}}$ with identical optimal solution value is then given by
			\begin{mini*}|s|[2]
				{\dualBudgetCap, \kappa, \rho}
				{\sum_{\demand \in \Demands} \left(\upperUnsteerable_\demand \dualBudgetCap_\demand - \lowerUnsteerable_\demand \kappa_\demand \right) + \demandBound_\indexUnsteerable \rho}
				{}
				{\LPDualAssumption{\hat{\varUnsteerableDemand}}\quad}
				\addConstraint{\dualBudgetCap_\demand - \kappa_\demand + \rho}{\geq \hat{\varUnsteerableDemand}_{\demand k}}{\quad \forall \demand \in \Demands}
				\addConstraint{\dualBudgetCap_\demand, \kappa_\demand, \rho}{\geq 0}{ \quad \forall \demand \in \Demands.}
			\end{mini*}
Substituting the dual problem back into \eqref{robAss1} and \eqref{robAss2}, we get the following linear set of constraints which enforce Assumption~\ref{ass:1} for $\GammaRobustBendersFormulation$:	
\begin{align} 
&\sum_{\demand \in \Demands} \left(\upperUnsteerable_\demand \dualBudgetCap^\location_\demand - \lowerUnsteerable_\demand \kappa^\location_\demand \right) + \demandBound_\indexUnsteerable \rho^\location \leq 	\capacitySession\, x_\location & \forall \location\in\Locations\\
&\sum_{\demand \in \Demands} \left(\upperUnsteerable_\demand \dualBudgetCap^\practice_\demand - \lowerUnsteerable_\demand \kappa^\practice_\demand \right) + \demandBound_\indexUnsteerable \rho^\practice \leq \capacityPractice_\practice  & \forall \practice\in\Practices\\
&\dualBudgetCap^k_\demand - \kappa^k_\demand + \rho^k \geq \hat{\varUnsteerableDemand}_{\demand k}  &\forall \demand \in \Demands,\, \forall k\in \Locations \cup \Practices\\
&\dualBudgetCap^k_\demand, \kappa^k_\demand, \rho^k \geq 0 &\forall \demand \in \Demands,\, \forall k\in \Locations \cup \Practices.
\end{align}
			
Last but not least, it remains to consider the enforcement of Assumption~\ref{ass:1} as we disaggregate sessions; see Section~\ref{sec:multi_period}.
For the session-specific strategic planning problem, the residual treatment capacities for fixed first-stage decisions $\hat{\varUnsteerableDemand}$ and $\hat{x}$ are given by

\begin{align*}
\helper_{\elocation}&\coloneqq\capacitySession\, \hat{x}_{\elocation} - 
\sum_{\edemand\in\eUnsteerableConsiderationSet(\elocation)}\unsteerableDemand_{\edemand}\, \hat{\varUnsteerableDemand}_{\edemand 
	\elocation} & \forall \elocation\in\eLocations,\\
\helper_{\epractice}&\coloneqq\capacityPractice_{\epractice} 
-\sum_{\edemand\in\eUnsteerableConsiderationSet(\epractice)}\unsteerableDemand_{\edemand}\,
\hat{\varUnsteerableDemand}_{\edemand 
	\epractice} & \forall \epractice\in\ePractices.  
\intertext{
As a result, we can enforce Assumption~\ref{ass:1} in $\SessionSpecificBendersFormulation$ using the constraints
}
&\sum_{\edemand\in\eUnsteerableConsiderationSet(\elocation)}\unsteerableDemand_{\edemand}\, \varUnsteerableDemand_{\edemand 
	\elocation} \leq 	\capacitySession\, x_{\elocation} & \forall \elocation\in\eLocations\\
&\sum_{\edemand\in\eUnsteerableConsiderationSet(\epractice)}\unsteerableDemand_{\edemand}\
\varUnsteerableDemand_{\edemand 
	\epractice} \leq \capacityPractice_{\epractice}  & \forall \epractice\in\ePractices.
\end{align*}

\section{Separation LP}
\label{appendix:separationLP}
The separation problem for $\BendersFormulation$ can be formulated as an LP based on the observations made in the proof of Theorem~\ref{thm:bendersoptcuts}.
That is, we need to decide whether for fixed first-stage decisions $\hat{x}$ and $\hat{\varUnsteerableDemand}$ there exists $U\subseteq \Demands$ such that $\sum_{\demand \in U} \steerableDemand_\demand  > \sum_{k \in \ConsiderationSet(U)} \helper_k$.
%; where $\helper_k \in \N$ denotes the constant residual treatment capacities at $k\in \Locations \cup \Practices$.
By encoding the choice of $U\subseteq \Demands$ through the variables $\varEncodeU_\demand \in [0,1]$ for all $\demand \in \Demands$ and the corresponding consideration set $\ConsiderationSet(U)$ through the variables $\varEncodeN_k\in [0,1]$ for all $k\in \Locations \cup \Practices$, we obtain the following formulation of the separation problem:
\begin{maxi*}|s|[2]
	{\varEncodeU, \varEncodeN}
	{\sum_{\demand \in \Demands} \steerableDemand_\demand\, \varEncodeU_\demand - \sum_{k \in \Locations \cup \Practices} \helper_k\, \varEncodeN_k}
	{}{\SeparationProblemFormulationAppendix\quad}
	\addConstraint{\varEncodeN_k}{\geq \varEncodeU_\demand}{\quad \forall \demand \in \Demands,\, \forall k \in \ConsiderationSet(\demand)}	
	\addConstraint{\varEncodeU_\demand}{\in  [0,1]}{\quad\forall \demand \in \Demands}
	\addConstraint{\varEncodeN_k}{\in  [0,1]}{\quad\forall k \in \Locations\cup \Practices.}
\end{maxi*}
Formulation $\SeparationProblemFormulationAppendix$ solves the dual problem to the LP-relation of the Benders subproblem  $\LPBendersSubproblem{\hat{y}}{\hat{x}}{\hat{w}}$.
If the optimal solution value to $\SeparationProblemFormulationAppendix$ is strictly positive, this yields a violated subset $\hat{U}\subseteq \Demands$ and we must resolve the restricted master problem.

Taking a closer look at formulation $\SeparationProblemFormulationAppendix$, it becomes evident that the separation problem for $\BendersFormulation$ is trivial if we only consider unsteerable demands: 
When $\steerableDemand_\demand = 0$ for all $\demand \in \Demands$ and $\helper_k \geq 0$ for all $k\in \Locations \cup \Practices$ due to Assumption~\ref{ass:1}, the objective of $\SeparationProblemFormulationAppendix$ is obviously non-positive.
Thus, the optimal solution value to $\SeparationProblemFormulationAppendix$ must be non-positive and there cannot exist a violated subset.

Note, that formulation $\SeparationProblemFormulationAppendix$ can be used to solve the separation problem for $\IntervalRobustBendersFormulation$ by simply substituting the deterministic steerable demands $\steerableDemand_\demand \in \N$ by the worst case uncertain steerable demands $\upperSteerable_\demand\in \N$ for all $\demand \in \Demands$.
	
	\section{Characteristics of test instances}
	\label{appendix:a}
	\input{tables/instances.tex}	
	
	\section{Computational results}
	\label{appendix:b}
	\input{tables/study.tex}

%%%%%%%%%%%%%%%%%%%%%%%%%
% Additional Appendices %
%%%%%%%%%%%%%%%%%%%%%%%%%
	
%	\section{Evaluation of operation cost}
%	\input{figures/GraphsW_full.tex}
%	\input{figures/GraphsD_full.tex}
%	\newpage
%	\section{Evaluation of solution quality based on SiM-Care scenarios}
%	\input{figures/Eval_Graphs_full.tex}
%	\section{Evaluation of solution quality based on SiM-Care scenarios with local surges in demand}
%	\input{figures/Eval_Graphs_infec_full.tex}
\end{appendices}

%% file: tables/instances.tex
\begin{table}[H]
	\centering
	\small
	\begin{tabular}{@{}lllrrrrrrr@{}}
		\toprule
		Inst. & $\maxDist$ & $\fracUnsteerable$ & $|\Demands|$ & $|\Locations|$ & $|\Practices|$ & $\sum \steerableDemand{+}\unsteerableDemand$ & $\sum \upperSteerable {+} \upperUnsteerable$ & $\demandBound_\indexSteerable$ & $\demandBound_\indexUnsteerable$ \\ \midrule
		\num{1} & \num{6} & \num{0.2} & \num{417} & \num{28} & \num{16} & \num{3888} & \num{6400} & \num{3233} & \num{808} \\
		\num{2} & \num{6} & \num{0.25} & \num{417} & \num{28} & \num{16} & \num{3888} & \num{6400} & \num{3031} & \num{1010} \\
		\num{3} & \num{6} & \num{0.3} & \num{417} & \num{28} & \num{16} & \num{3888} & \num{6400} & \num{2829} & \num{1212} \\
		\num{4} & \num{6} & \num{0.35} & \num{417} & \num{28} & \num{16} & \num{3888} & \num{6400} & \num{2627} & \num{1414} \\
		\num{5} & \num{6} & \num{0.4} & \num{417} & \num{28} & \num{16} & \num{3888} & \num{6400} & \num{2425} & \num{1616} \\
		\num{6} & \num{6} & \num{0.45} & \num{417} & \num{28} & \num{16} & \num{3888} & \num{6400} & \num{2223} & \num{1818} \\
		\addlinespace[.2cm]
		\num{7} & \num{7} & \num{0.2} & \num{462} & \num{28} & \num{16} & \num{3886} & \num{6517} & \num{3233} & \num{808} \\
		\num{8} & \num{7} & \num{0.25} & \num{462} & \num{28} & \num{16} & \num{3886} & \num{6517} & \num{3031} & \num{1010} \\
		\num{9} & \num{7} & \num{0.3} & \num{462} & \num{28} & \num{16} & \num{3886} & \num{6517} & \num{2829} & \num{1212} \\
		\num{10} & \num{7} & \num{0.35} & \num{462} & \num{28} & \num{16} & \num{3886} & \num{6517} & \num{2627} & \num{1414} \\
		\num{11} & \num{7} & \num{0.4} & \num{462} & \num{28} & \num{16} & \num{3886} & \num{6517} & \num{2425} & \num{1616} \\
		\num{12} & \num{7} & \num{0.45} & \num{462} & \num{28} & \num{16} & \num{3886} & \num{6517} & \num{2223} & \num{1818} \\
		\addlinespace[.2cm]
		\num{13} & \num{8} & \num{0.2} & \num{506} & \num{28} & \num{16} & \num{3885} & \num{6668} & \num{3233} & \num{808} \\
		\num{14} & \num{8} & \num{0.25} & \num{506} & \num{28} & \num{16} & \num{3885} & \num{6668} & \num{3031} & \num{1010} \\
		\num{15} & \num{8} & \num{0.3} & \num{506} & \num{28} & \num{16} & \num{3885} & \num{6668} & \num{2829} & \num{1212} \\
		\num{16} & \num{8} & \num{0.35} & \num{506} & \num{28} & \num{16} & \num{3885} & \num{6668} & \num{2627} & \num{1414} \\
		\num{17} & \num{8} & \num{0.4} & \num{506} & \num{28} & \num{16} & \num{3885} & \num{6668} & \num{2425} & \num{1616} \\
		\num{18} & \num{8} & \num{0.45} & \num{506} & \num{28} & \num{16} & \num{3885} & \num{6668} & \num{2223} & \num{1818} \\
		\addlinespace[.2cm]
		\num{19} & \num{9} & \num{0.2} & \num{546} & \num{28} & \num{16} & \num{3888} & \num{6829} & \num{3233} & \num{808} \\
		\num{20} & \num{9} & \num{0.25} & \num{546} & \num{28} & \num{16} & \num{3888} & \num{6829} & \num{3031} & \num{1010} \\
		\num{21} & \num{9} & \num{0.3} & \num{546} & \num{28} & \num{16} & \num{3888} & \num{6829} & \num{2829} & \num{1212} \\
		\num{22} & \num{9} & \num{0.35} & \num{546} & \num{28} & \num{16} & \num{3888} & \num{6829} & \num{2627} & \num{1414} \\
		\num{23} & \num{9} & \num{0.4} & \num{546} & \num{28} & \num{16} & \num{3888} & \num{6829} & \num{2425} & \num{1616} \\
		\num{24} & \num{9} & \num{0.45} & \num{546} & \num{28} & \num{16} & \num{3888} & \num{6829} & \num{2223} & \num{1818} \\
		\addlinespace[.2cm]
		\num{25} & \num{10} & \num{0.2} & \num{576} & \num{28} & \num{16} & \num{3887} & \num{6938} & \num{3233} & \num{808} \\
		\num{26} & \num{10} & \num{0.25} & \num{576} & \num{28} & \num{16} & \num{3887} & \num{6938} & \num{3031} & \num{1010} \\
		\num{27} & \num{10} & \num{0.3} & \num{576} & \num{28} & \num{16} & \num{3887} & \num{6938} & \num{2829} & \num{1212} \\
		\num{28} & \num{10} & \num{0.35} & \num{576} & \num{28} & \num{16} & \num{3887} & \num{6938} & \num{2627} & \num{1414} \\
		\num{29} & \num{10} & \num{0.4} & \num{576} & \num{28} & \num{16} & \num{3887} & \num{6938} & \num{2425} & \num{1616} \\
		\num{30} & \num{10} & \num{0.45} & \num{576} & \num{28} & \num{16} & \num{3887} & \num{6938} & \num{2223} & \num{1818} \\
		\addlinespace[.2cm]
		\num{31} & \num{11} & \num{0.2} & \num{583} & \num{28} & \num{16} & \num{3889} & \num{6960} & \num{3233} & \num{808} \\
		\num{32} & \num{11} & \num{0.25} & \num{583} & \num{28} & \num{16} & \num{3889} & \num{6960} & \num{3031} & \num{1010} \\
		\num{33} & \num{11} & \num{0.3} & \num{583} & \num{28} & \num{16} & \num{3889} & \num{6960} & \num{2829} & \num{1212} \\
		\num{34} & \num{11} & \num{0.35} & \num{583} & \num{28} & \num{16} & \num{3889} & \num{6960} & \num{2627} & \num{1414} \\
		\num{35} & \num{11} & \num{0.4} & \num{583} & \num{28} & \num{16} & \num{3889} & \num{6960} & \num{2425} & \num{1616} \\
		\num{36} & \num{11} & \num{0.45} & \num{583} & \num{28} & \num{16} & \num{3889} & \num{6960} & \num{2223} & \num{1818} \\ \bottomrule
	\end{tabular}
		\caption{Test instances with their characteristics.}
		\label{table:instances}
\end{table}

%% file: tables/study.tex
% Please add the following required packages to your document preamble:
% \usepackage{booktabs}
% \usepackage{multirow}
\begin{table}[H]
	\centering
	\small
	\begin{tabular}{@{}LLLRRRRRR@{}}
		\toprule
		\multicolumn{3}{c}{Instance} & \multicolumn{3}{c}{Objective} & \multicolumn{3}{c}{CPU [s]} \\ 
		\cmidrule(r){1-3} \cmidrule(lr){4-6}   \cmidrule(l){7-9} 
		\# & \maxDist & \fracUnsteerable & \BendersFormulationNB & \GammaRobustBendersFormulationNB  & \IntervalRobustBendersFormulationNB & \BendersFormulationNB & \GammaRobustBendersFormulationNB & \IntervalRobustBendersFormulationNB \\
		\midrule
		1 & 6 & 0.2 & 31 & 46 & 64 & 6 & 32 & 5 \\
		2 & 6 & 0.25 & 31 & 48 & 64 & 8 & 44 & 3 \\
		3 & 6 & 0.3 & 32 & 50 & 64 & 7 & 45 & 5 \\
		4 & 6 & 0.35 & 34 & 53 & 64 & 5 & 96 & 5 \\
		5 & 6 & 0.4 & 35 & 55 & 64 & 3 & 139 & 8 \\
		6 & 6 & 0.45 & 36 & 67 & 70 & 5 & 93 & 12 \\
		\addlinespace[.2cm]
		7 & 7 & 0.2 & 25 & 32 & 68 & 10 & 26 & 3 \\
		8 & 7 & 0.25 & 26 & 37 & 68 & 6 & 38 & 5 \\
		9 & 7 & 0.3 & 27 & 39 & 68 & 5 & 26 & 6 \\
		10 & 7 & 0.35 & 29 & 43 & 68 & 7 & 26 & 7 \\
		11 & 7 & 0.4 & 30 & 48 & 68 & 5 & 57 & 3 \\
		12 & 7 & 0.45 & 31 & 52 & 68 & 3 & 18 & 4 \\
		\addlinespace[.2cm]
		13 & 8 & 0.2 & 17 & 26 & 75 & 8 & 73 & 12 \\
		14 & 8 & 0.25 & 17 & 30 & 75 & 6 & 42 & 4 \\
		15 & 8 & 0.3 & 20 & 36 & 75 & 9 & 51 & 9 \\
		16 & 8 & 0.35 & 20 & 40 & 75 & 9 & 41 & 7 \\
		17 & 8 & 0.4 & 21 & 45 & 75 & 7 & 44 & 34 \\
		18 & 8 & 0.45 & 22 & 54 & 75 & 4 & 19 & 6 \\
		\addlinespace[.2cm]
		19 & 9 & 0.2 & 9 & 13 & 81 & 2 & 18 & \num{1678} \\
		20 & 9 & 0.25 & 10 & 15 & 81 & 4 & 16 & 683 \\
		21 & 9 & 0.3 & 11 & 21 & 81 & 5 & 29 & \num{1936} \\
		22 & 9 & 0.35 & 11 & 27 & 81 & 4 & 30 & \num{1847} \\
		23 & 9 & 0.4 & 12 & 32 & 81 & 3 & 49 & \num{1069} \\
		24 & 9 & 0.45 & 13 & 47 & 81 & 4 & 29 & 24 \\
		\addlinespace[.2cm]
		25 & 10 & 0.2 & 3 & 4 & 87 & 4 & 14 & 4 \\
		26 & 10 & 0.25 & 3 & 11 & 87 & 2 & 10 & \num{6944} \\
		27 & 10 & 0.3 & 3 & 12 & 87 & 4 & 10 & \num{5993} \\
		28 & 10 & 0.35 & 4 & 15 & 87 & 2 & 8 & \num{8577} \\
		29 & 10 & 0.4 & 4 & 32 & 87 & 2 & 16 & \num{1126} \\
		30 & 10 & 0.45 & 11 & 41 & 87 & 4 & 10 & 191 \\
		\addlinespace[.2cm]
		31 & 11 & 0.2 & 0 & 0 & 88 & 2 & 8 & \num{10026} \\
		32 & 11 & 0.25 & 0 & 7 & 88 & 2 & 9 & 4 \\
		33 & 11 & 0.3 & 0 & 8 & 88 & 2 & 12 & \num{12815} \\
		34 & 11 & 0.35 & 0 & 15 & 88 & 2 & 20 & \num{12925} \\
		35 & 11 & 0.4 & 0 & 32 & 88 & 2 & 12 & \num{4968} \\
		36 & 11 & 0.45 & 7 & 41 & 88 & 4 & 13 & \num{1403}\\
		\bottomrule
	\end{tabular}
	\caption{Computational results.}
	\label{table:comp}
\end{table}